\documentclass{amsart}

\newtheorem{theorem}{Theorem}[section]
\newtheorem{lemma}[theorem]{Lemma}
\newtheorem{remark}[theorem]{Remark}
\newtheorem{proposition}[theorem]{Proposition}
\newtheorem{corollary}[theorem]{Corollary}
\newtheorem{example}[theorem]{Example}
\newtheorem{definition}[theorem]{Definition}

\usepackage{amsmath, amsthm, amssymb}
\usepackage{mathtools}
\usepackage{hyperref}
\usepackage{enumerate}
\usepackage{tikz}

\newcommand{\cemph}[1]{\emph{\color{red}#1}}

\newcommand\R{\mathbb{R}}
\newcommand\N{\mathbb{N}}
\newcommand\B{\mathbb{B}}
\newcommand\mS{\mathbb{S}}
\newcommand\CK{\mathcal{K}}
\newcommand\CO{\mathcal{O}}

\DeclareMathOperator{\conv}{conv}
\DeclareMathOperator{\pos}{pos}
\DeclareMathOperator{\aff}{aff}
\DeclareMathOperator{\lin}{lin}
\DeclareMathOperator{\cl}{cl}
\DeclareMathOperator{\inte}{int}
\DeclareMathOperator{\bd}{bd}
\DeclareMathOperator{\ext}{ext}
\DeclareMathOperator{\ray}{ray}

\newcommand\gauge[1]{\left\| #1 \right\|}
\newcommand\normal[2]{\frac{#1}{\left\| #1 \right\|_{#2}}}

\newcommand\upmean[3]{\overline{M}_{#1}(#2,#3)}
\newcommand\lowmean[3]{\underline{M}_{#1}(#2,#3)}

\newcommand\overM{\overline{M}}
\newcommand\underM{\underline{M}}

\newcommand\ds{2/\scale pt}

\newcommand{\stacked}[3]{%
  \begingroup
  \settowidth{\dimen0}{$#1$}%
  \vbox{%
    \hbox to \dimen0{\hfil $#1$ \hfil}%
    \vskip 0em
    \hbox to \dimen0{\hfil $#2$ \hfil}%
    \vskip -0.3em
    \hbox to \dimen0{\hfil $#3$ \hfil}%
  }%
  \endgroup
}

\DeclareRobustCommand{\m&}{%
  \ifdim\fontdimen1\font>0pt
    \textsl{\symbol{`\&}}%
  \else
    \symbol{`\&}%
  \fi
}

\newcommand\mpvara{x}
\newcommand\mpvarb{y}
\newcommand\intvar{t}

\title[$p$-Means of Convex Bodies]{$p$-Means of Convex Bodies: Sharpening Relations and Structural Properties}
\author[R. Brandenberg]{Ren\'{e} Brandenberg}
\author[F. Grundbacher]{Florian Grundbacher}
\date{\today}

\begin{document}

\begin{abstract}
We study general $p$-means of convex bodies, extending the classical definitions by W.~J.~Firey via support and gauge functions to two families ranging over all $p \in [-\infty,\infty]$. For values of $p$ beyond the classical ranges, we show that $p$-means of polytopes are again polytopes, yielding simpler structural descriptions. Using a natural characterization of dilates of convex bodies based on their boundary structure, we characterize the equality cases between the two types of $p$-means for the same $p$-value. Extending recent results on standard mean-symmetrizations of convex bodies, we further establish (in almost all instances tight) inequalities quantifying how well arbitrary $p$-means of convex bodies approximate each other. These bounds lead to characterizations and sharp stability results for the equality cases between $p$-means for different $p$-values. As a corollary, every Minkowski centered convex body is equidistant from all its $p$-symmetrizations with respect to the Banach--Mazur distance.
\end{abstract}

\maketitle

\section{Introduction and Results}

Since the introduction of $L_p$-addition by Firey in \cite{Fipmeans}
and its application to Brunn--Minkowski theory by Lutwak in \cite{Lutheo1,Lutheo2},
the resulting $L_p$-Brunn--Minkowski theory has been an essential field of study in convex geometry.
A general overview of the topic is presented in \cite[Chapter 9]{Sch}.
See also \cite{Bo} and the references therein for more recent developments.

Especially in light of the log-Brunn--Minkowski conjecture posed in \cite{BoLuYaZh},
the question of how the geometric mean of convex bodies should be defined (re-)emerged,
both as a tool for attacking the conjecture (cf.~\cite{CrFr}) 
and to provide a mean that fulfills a list of desired properties.
For the latter goal, multiple approaches have been presented in recent literature
(cf.~\cite{MiRoirrat,MiRog1,MiRoenvelopes,Ro} and the references therein).
While the constructions provided constitute insightful ideas,
the authors of these papers pointed out themselves that there is (almost) always some desired property of geometric means missing.
Moreover, some of the constructions appear somewhat difficult to handle,
such that it is still an open question
whether the properties collected in \cite{Ro}
characterize a unique geometric mean
despite multiple constructions being known to satisfy them
at least for symmetric convex bodies.

The above questions about the definition and the uniqueness of the geometric mean of convex bodies motivate
a closer analysis of the relations between general $p$-means of convex bodies.
Most constructions of the geometric mean are naturally based on $p$-means, which is especially true for those in \cite{MiRoirrat,MiRog1,Ro}.
An improved understanding of how $p$-means are related to each other
therefore provides a better framework to deal with the above questions.
The success of this approach will be demonstrated
in a follow-up publication focused on the geometric mean.

Before we give some more intrinsic motivation for studying the relations between $p$-means of convex bodies, we introduce the necessary notation for the following discussion.
Given $X,Y \subset \R^n$, $z \in \R^n$, $\Lambda \subset \R$, and $\mu \in \R$,
we define the \cemph{Minkowski sum}
$X + Y \coloneqq \{x + y : x \in X, y \in Y\}$,
the \cemph{$z$-translation}
$X + z \coloneqq z + X \coloneqq \{z\} + X$,
the \cemph{$\mu$-dilation} $\mu X \coloneqq \{ \mu x : x \in X \}$,
the \cemph{$\Lambda$-dilation-union}
$\Lambda X \coloneqq \bigcup_{\lambda \in \Lambda} \lambda X$,
and $\Lambda z \coloneqq \Lambda \{ z \}$.
We abbreviate $(-1) X$ and $X + (-1) Y$ by $-X$ and $X - Y$, respectively.
The \cemph{convex}, \cemph{affine}, \cemph{linear}, and \cemph{positive hulls} of $X$ are denoted by $\conv(X)$, $\aff(X)$, $\lin(X)$, and $\pos(X)$, respectively.

For any closed and convex set $C \subset \R^n$ with $0 \in C$, its \cemph{polar} is defined as
\[
    C^\circ
    \coloneqq \{ a \in \R^n : a^T x \leq 1 \text{ for all } x \in C \}.
\]
Its
\cemph{support function} is given by
\[
    h_C \colon \R^n \to [0,\infty],
    h_C(a) \coloneqq \sup \{ a^T x : x \in C \},
\]
and its \cemph{gauge function} by
\[
    \gauge{\cdot}_C \colon \R^n \to [0,\infty],
    \gauge{x}_C \coloneqq \inf \{ \lambda \geq 0 : x \in \lambda C \}.
\]

We write $\CK^n$ for the set
of \cemph{(convex) bodies} in~$\R^n$, i.e., non-empty compact convex sets,
and $\CK^n_0$ for the set of bodies containing the origin in their interior.
For $p \geq 1$, the $n$-dimensional \cemph{unit ball}
of the \cemph{$p$-norm} $\gauge{\cdot}_p$
is given by $\B_p^n \in \CK^n_0$ with \cemph{unit sphere} $\mS_p^n$,
where we omit the dimension if it is clear from context.
For general $p \in \R \setminus \{0\}$,
the \cemph{$p$-mean} of two positive real numbers $\mpvara, \mpvarb > 0$
is defined as
\[
	m_p(\mpvara,\mpvarb)
    \coloneqq \left( \frac{\mpvara^p + \mpvarb^p}{2} \right)^{\frac{1}{p}}.
\]
Moreover, the \cemph{$(-\infty)$-}, \cemph{$0$-}, and \cemph{$\infty$-means} are, as continuous extensions in $p$, given by
\[
	m_{-\infty}(\mpvara,\mpvarb) \coloneqq \min \{ \mpvara, \mpvarb \},
        \quad
	m_0(\mpvara,\mpvarb) \coloneqq \sqrt{\mpvara \mpvarb},
        \quad \text{and} \quad
	m_{\infty}(\mpvara,\mpvarb) \coloneqq \max \{ \mpvara, \mpvarb \}.
\]

The \cemph{standard means} of convex bodies $K, L \in \CK^n_0$ are the \cemph{minimum} $K \cap L$,
the \cemph{harmonic mean} $( \frac{K^\circ+L^\circ}{2} )^\circ$,
the \cemph{arithmetic mean} $\frac{K+L}{2}$,
and the \cemph{maximum} $\conv(K \cup L)$.
These sets are increasing in order
as first shown in \cite{Fipolarmeans}.
Calling the second set the harmonic mean of $K$ and $L$ is justified by interpreting the polar of a convex body as its inverse
(cf.~\cite{MiRog1}).

With regard to the geometric mean, we are especially interested in $p$-means of convex bodies when $p=0$.
However, due to the behavior of $p$-means in dependence of $p$, they are often defined only for restricted ranges of $p$:
Starting with the observation that $h_{\frac{K+L}{2}} = \frac{1}{2} (h_K+h_L)$, Firey defined the first type of $p$-mean implicitly in \cite{Fipmeans} by setting its support function to be $m_p(h_K,h_L)$.
This is possible whenever $m_p(h_k,h_L)$ is a convex function,
which is guaranteed if $p \in [1,\infty]$.
For $p < 1$, however, this convexity typically does not hold. 
Similarly, Firey generalized in \cite{Fipolarmeans} the property that $\gauge{\cdot}_{( \frac{K^\circ+L^\circ}{2} )^\circ} = \frac{1}{2} (\gauge{\cdot}_K+\gauge{\cdot}_L)$ to a second type of $p$-mean by defining it implicitly to have the gauge function $m_{-p}(\gauge{\cdot}_K,\gauge{\cdot}_L)$.
Again, owed to the required convexity of gauge functions, this is generally only possible for $p \in [-\infty,-1]$.
Since neither definition of $p$-means is applicable for $p=0$, a different approach is required.

\begin{definition}
\label{def:p-means}
Let $K,L \in \CK^n_0$ and $p \in [-\infty,\infty]$.
The \cemph{upper $p$-mean} of $K$ and $L$ is defined as
\[
    \upmean{p}{K}{L}
    \coloneqq \{ x \in \R^n :
    a^T x \leq m_p(h_K(a),h_L(a))
    \textup{ for all } a \in \mS_2 \}.
\]
Their \cemph{lower $p$-mean} is given by
\[
    \lowmean{p}{K}{L}
    \coloneqq \conv \left( \left\{
    m_p(\gauge{v}_K^{-1},\gauge{v}_L^{-1}) v :
    v \in \mS_2 \right\} \right).
\]
\end{definition}

It is immediate to see that $\upmean{p}{K}{L}$ is the largest body $C$ that satisfies $h_C \leq m_p(h_K,h_L)$, whereas $\lowmean{p}{K}{L}$ is the smallest body $C$ with $\gauge{\cdot}_C \leq m_{-p} (\gauge{\cdot}_K,\gauge{\cdot}_L)$.
In particular, they coincide with the $p$-means proposed by Firey for $p \geq 1$ and $p \leq -1$, respectively, which includes the four standard means. Since one typically wants the $p$-means of $K$ and $L$ to generalize these four special cases,
it is natural to define \emph{the} $p$-mean of $K$ and $L$ as $\upmean{p}{K}{L}$ when $p > 0$ and as $\lowmean{p}{K}{L}$ when $p < 0$.
Nonetheless, we study upper and lower $p$-means for all values $p \in [-\infty,\infty]$
and for brevity refer to the entirety of both series simply as $p$-means.
A major difficultly in studying either type of $p$-mean
for the values of $p$ that Firey left out lies in the fact that the inequalities on their support or gauge functions may in general be strict.

The modification in the definition of the upper $p$-means was first suggested in \cite{BoLuYaZh}, there only for $p \geq 0$.
It has since gained some attention, largely motivated through the log-Brunn--Minkowski conjecture (see, e.g., \cite{ChHuLiLiu,Pu}). More rarely, also the case $p < 0$ has been considered (see, e.g., \cite{KolMi}).

For the lower $p$-means, on the other hand, typically the slightly different concept of $p$-radial sums is used (cf.~\cite[Chapter~$9$]{Sch}).
Replacing gauge functions by radial functions, those sums combine two given star-bodies into the smallest star-body that contains all points in the generating set of the corresponding lower $p$-mean.
Since we want, especially with regard to possible definitions of the geometric mean, that $p$-means of convex bodies are again convex, we additionally take convex hulls here.
An advantage of this approach is that upper and lower $p$-means are directly related by polarity, which allows for a more unified theory and enables us to transfer results between the two types of $p$-means.
This is especially useful when results are easier established for one specific type of $p$-mean and is taken advantage of in both directions.
Let us also point out that lower $p$-means have been considered for example in \cite{HCYN,Sa}, even if only indirectly, where dual (log-)Brunn--Minkowski-type inequalities have been established.

An early, yet central result by Firey about $p$-means of convex bodies is the generalization of the inequality between the arithmetic mean and the harmonic mean to the realm of convex bodies.
Given the importance of this inequality already for real numbers, it is only natural that also its analog for convex bodies $K,L \in \CK^n_0$ has received special attention.
With $p,q \in [1,\infty]$, $p \leq q$, it takes the form
\begin{align}
\begin{split} \label{eq:Firey_chain}
    K \cap L
    & = \lowmean{-\infty}{K}{L}
    \subset \lowmean{-q}{K}{L}
    \subset \lowmean{-p}{K}{L}
    \subset \lowmean{-1}{K}{L}
    = \left( \frac{K^\circ + L^\circ}{2} \right)^\circ
    \\
    & \subset \frac{K+L}{2}
    = \upmean{1}{K}{L}
    \subset \upmean{p}{K}{L}
    \subset \upmean{q}{K}{L}
    \subset \upmean{\infty}{K}{L}
    = \conv(K \cup L).
\end{split}
\end{align}
Recently, various aspects of \eqref{eq:Firey_chain} have been sharpened in \cite{BrDiMe} for the standard means,
with a particular focus on the symmetrizations of convex bodies, i.e., when $L = -K$.
Most importantly to us, the optimality of the natural containments is studied,
and additional dilation factors are introduced to improve the containments whenever possible.
Based on this idea, the chain \eqref{eq:Firey_chain} is also reversed in the order of $p$
by again introducing optimal dilation factors.

Generalizing these results to the setting of $p$-means of arbitrary bodies marks our second starting motivation.
It turns out that several of the insights in \cite{BrDiMe} remain true in this more general situation,
though some phenomena are lost along the way.
We give a detailed comparison of our new results to those in \cite{BrDiMe} in Remark~\ref{rem:BDM_comparison} below.
The advantage of the generalized setting is that it provides a
much less restricted way of studying the relations between arbitrary $p$-means.
In particular, we obtain a straightforward proof of the equality cases between $p$-means and the stability thereof,
and gain insights into the structures within the Banach--Mazur compactum that $p$-means provide.

Before we are able to concentrate on these finer structural results, 
we first need to address some more elementary aspects.
Given that the two types of $p$-means have hardly been studied for the same $p$-values at the same time,
many questions about even their basic relations are open.
A first natural question is, for example, under which circumstances they coincide.
It turns out that the two types of $p$-means are almost always different, agreeing only in some trivial cases.

\begin{theorem}
\label{thm:same_p}
Let $K,L \in \CK^n_0$ and $p \in [-\infty,\infty]$.
Then
\[
    \lowmean{p}{K}{L}
    \subset \upmean{p}{K}{L},
\]
with equality if and only if $n=1$, or $p = \pm \infty$, or $K$ and $L$ are dilates of each other.
\end{theorem}

The same questions about set containments and equality cases can be asked about $p$-means for different values of $p$.
It is clear from the monotonicity of $p$-means of real numbers in $p$ that the upper and lower $p$-means are individually increasing in $p$.
Moreover, Theorem~\ref{thm:same_p} shows how the individual chains are compatible with each other, leading to the following connected double-chain for $-\infty < p < q < \infty$.

\begin{equation}
\begin{split}\label{eq:forward_double_chain}
    \stacked{\upmean{-\infty}{K}{L}}{\rotatebox{90}{$=$}}{\lowmean{-\infty}{K}{L}} \
    \stacked{\subset}{}{\subset} \
    \stacked{\upmean{p}{K}{L}}{\rotatebox{90}{$\subset$}}{\lowmean{p}{K}{L}} \
    \stacked{\subset}{}{\subset} \
    \stacked{\upmean{q}{K}{L}}{\rotatebox{90}{$\subset$}}{\lowmean{q}{K}{L}} \
    \stacked{\subset}{}{\subset} \
    \stacked{\upmean{\infty}{K}{L}}{\rotatebox{90}{$=$}}{\lowmean{\infty}{K}{L}}
\end{split}
\end{equation}

Extending Firey's results in \cite{Fipolarmeans,Fipmeans},
we find that,
just like for real numbers,
either type of $p$-mean coincides with
either type of $q$-mean for any $p \neq q$ if and only if $K = L$.
In other words, any two means that are not
directly below each other in \eqref{eq:forward_double_chain}
coincide precisely when $K = L$.
In fact, we can show even a stability version of this result.
Doing so requires an appropriate metric.
Typical choices might be the Hausdorff distance, the \emph{geometric distance} from \cite{LiTJ},
or volume-based distances analogous to the parameters in \cite{Tasch}.
However, their individual properties make them less suitable for our purposes.
Instead, we draw inspiration from the methods in \cite{BrDiMe} and use the following distance measure, which unifies some of the desired properties of the aforementioned distances while overcoming their individual disadvantages.
We refer the reader to Section~\ref{sec:radii} for a
detailed discussion of the properties of the following metric.

\begin{definition}
Let $K,L \subset \R^n$.
The \cemph{covering radius} of $K$ w.r.t.~$L$ is defined as
\[
    R_0(K,L)
    \coloneqq \inf \{ \lambda \geq 0 : K \subset \lambda L \}
    \in [0,\infty].
\]
The \cemph{covering distance} between $K$ and $L$ is $R_0^{\max}(K,L) \coloneqq \max \{ R_0(K,L), R_0(L,K) \}$.
\end{definition}

We are now ready to state our stability version of the equality cases between $p$-means for different $p$-values.
Asymptotic orders are always understood in terms of $0 \leq \varepsilon \to 0$.

\begin{theorem}
\label{thm:stability}
Let $K,L \in \CK^n_0$, $p,q \in [-\infty,\infty]$ with $p < q$,
and $M_p \in \{ \overline{M}_p, \underline{M}_ p\}$,
$M_q \in \{ \overline{M}_q, \underline{M}_q \}$.
Further write $\varepsilon \coloneqq R_0^{\max}(M_p(K,L),M_q(K,L)) - 1$.
\begin{enumerate}[(i)]
\item If $p \neq -\infty$, $q \neq \infty$, and $\varepsilon \leq \frac{m_q(c,1)}{m_p(c,1)} - 1$ for $c \coloneqq 1 + \frac{\delta}{\max\{|p|,|q|\}+3/2}$ where $\delta \in [0,1)$, then
\[
    R_0^{\max}(K,L)
    \leq 1 + \sqrt{\frac{8 \varepsilon}{(1-\delta) (q-p)}}.
\]
Moreover, for any $c' \geq 1$ and $\varepsilon' = \frac{m_q(c',1)}{m_p(c',1)} - 1$ there exist $K',L' \in \CK^n_0$ with
\[
    \varepsilon'
    = R_0^{\max}(M_p(K',L'),M_q(K',L')) - 1
        \quad \text{and} \quad
    1 + \sqrt{\frac{8 \varepsilon'}{q-p}}
    \leq R_0^{\max}(K',L').
\]

\item If $p = -\infty$ and $q \neq \infty$, or $p \neq -\infty$ and $q = \infty$, then
\[
    R_0^{\max}(K,L)
    = 1 + 2 \varepsilon + \CO(\varepsilon^2),
\]
where the error term depends on $p$ and $q$ but not on $K$, $L$, or $n$.

\item If $p = -\infty$ and $q = \infty$, then
\[
    R_0^{\max}(K,L)
    = 1 + \varepsilon.
\]
\end{enumerate}
\end{theorem}

The exact value of the error term in (ii) can be obtained from our proof of the theorem.

Our main approach to proving Theorem~\ref{thm:stability} is to extend the aforementioned study from \cite{BrDiMe}, where dilation factors are introduced to tighten the set inequalities between the standard means, to the setting of general $p$-means.
We begin with a general lower bound on these factors, which already yields the above stability estimate up to some analytical details.

\begin{theorem}
\label{thm:lower_bound}
Let $K,L \in \CK^n_0$, $p,q \in [-\infty,\infty]$,
and $M_p \in \left\{ \overM_p, \underM_p \right\}$, $M_q \in \left\{ \overM_q, \underM_q \right\}$.
Then
\[
    \frac{m_p(R_0^{\max}(K,L),1)}{m_q(R_0^{\max}(K,L),1)}
    \leq R_0(M_p(K,L),M_q(K,L)).
\]
Equality is attained whenever $n=1$ or $K$ and $L$ are dilates of each other.
\end{theorem}

We derive this lower bound in a straightforward way from upper bounds on the covering radii between $p$-means,
more precisely from part (i) in the theorem below.
To complete the study of covering radii between $p$-means (as far as we were able to),
we also give the upper bounds in the other parts of that theorem.
Note that we have in summary sharp upper and lower bounds
for almost all combinations of $p$-means and values of $R_0^{\max}(K,L)$,
the only exception being the upper bound on $R_0(\upmean{p}{K}{L},\lowmean{q}{K}{L})$ when $(p,q) \in (-\infty,1) \times (-1,\infty)$
(and a more explicit expression of the right-hand side in (iv) below).

\begin{theorem}
\label{thm:upper_bounds}
Let $K,L \in \CK^n_0$, $p,q \in [-\infty,\infty]$,
and $M_p \in \left\{ \overM_p, \underM_p \right\}$, $M_q \in \left\{ \overM_q, \underM_q \right\}$.
\begin{enumerate}[(i)]
\item If $(M_p,M_q) \neq (\overM_p,\underM_q)$ and $p \geq q$, then
\begin{align*}
    R_0(M_p(K,L),M_q(K,L))
    & = \frac{m_p(R_0^{\max}(K,L),1)}{m_q(R_0^{\max}(K,L),1)}.
\intertext{
\item If $(M_p,M_q) \neq (\overM_p,\underM_q)$ and $p < q$, then
}
    R_0(M_p(K,L),M_q(K,L))
    & \leq 1.
\intertext{
\item If $(p,q) \notin (1,\infty) \times (-\infty,-1)$, then
}
    R_0(\upmean{p}{K}{L},\lowmean{q}{K}{L})
    & \leq \min \{ m_p(R_0^{\max}(K,L),1), m_{-q}(R_0^{\max}(K,L),1) \}.
\intertext{
\item If $(p,q) \in (1,\infty) \times (-\infty,-1)$, then
}
    R_0(\upmean{p}{K}{L},\lowmean{q}{K}{L})
    & \leq 2^{\frac{1}{q} - \frac{1}{p}} 
        R_0 \left( \begin{pmatrix} R_0^{\max}(K,L) & 1 \\
                    1 & R_0^{\max}(K,L) \end{pmatrix}
                \B_{\frac{p}{p-1}}^2, \B_{-q}^2 \right).
\end{align*}
If $\frac{p}{p-1} \geq -q$, then the right-hand side is equal to $m_1(R_0^{\max}(K,L),1) = \frac{R_0^{\max}(K,L) + 1}{2}$.
\end{enumerate}
For any inequality exluding (possibly) the case $(p,q) \in  (-\infty,1) \times (-1,\infty)$ for (iii),
there exists for $n \geq 2$ and $R \geq 1$ a body $K \in \CK^n_0$ such that $R = R_0^{\max}(K,-K)$ and equality is attained for the choice $L = -K$.
\end{theorem}

Analogous to the results in \cite{BrDiMe}, part (i) of the above theorem
enables us to reverse the natural order of $p$-means,
leading to the following connected double-chain with dilation factors for $-\infty < p < q < \infty$.
For readability, we abbreviate $R_0^{\max} \coloneqq R_0^{\max}(K,L)$.
\begin{equation}
\begin{split} \label{eq:reverse_chain}
    \stacked{\upmean{\infty}{K}{L}}{\rotatebox{90}{$=$}}{\lowmean{\infty}{K}{L}} \
    \stacked{\subset}{}{\subset} \
    \stacked{m_{-q}(R_0^{\max},1) \upmean{q}{K}{L}}{\rotatebox{90}{$\subset$}}{m_{-q}(R_0^{\max},1) \lowmean{q}{K}{L}} \
    \stacked{\subset}{}{\subset} \
    \stacked{m_{-p}(R_0^{\max},1) \upmean{p}{K}{L}}{\rotatebox{90}{$\subset$}}{m_{-p}(R_0^{\max},1) \lowmean{p}{K}{L}} \
    \stacked{\subset}{}{\subset} \
    \stacked{R_0^{\max} \upmean{-\infty}{K}{L}}{\rotatebox{90}{$=$}}{R_0^{\max} \lowmean{-\infty}{K}{L}}
\end{split}
\end{equation}
The left-most and right-most bodies, and therefore all of them,
have a common boundary point.

Our examples to verify sharpness in the applicable cases of Theorem~\ref{thm:upper_bounds}~(iii),
as well as for studying the (strict) monotonicity of $p$-means in $K$ and $L$,
are based on the following result about $p$-means of polytopes.
It states that these $p$-means are again polytopes
for the ranges of $p$-values where the implicit definitions by Firey
via support and gauge functions generally 
fail.
While our applications of this result show how it can greatly simplify the computation of the corresponding $p$-means,
it is also a structural insight of its own interest,
particularly for $\upmean{p}{K}{L}$ when $p \in [0,1]$ with regard to the log-Brunn--Minkowski conjecture.

Let us introduce some notation to properly state the result.
For $K \in \CK^n$, we denote its set of \cemph{extreme points} by $\ext(K)$.
For $x \in K$, $N_K(x)$ is the \cemph{cone of outer normals} of $x$ w.r.t.~$K$.
For a polytope $P \in \CK^n_0$,
we define its \cemph{coarse normal fan} as
$\mathcal{N}_P \coloneqq \{ N_P(x) : x \in \ext(P) \}$
and its \cemph{coarse face fan} as
$\mathcal{F}_P \coloneqq \{ \pos(F) : F \text{ is a facet of } P \}.$
Note that $\mathcal{N}_P = \mathcal{F}_{P^\circ}$ and $\mathcal{F}_P = \mathcal{N}_{P^\circ}$.
For a cone $C \subset \R^n$, we write $\ray(C)$ for the set of Euclidean unit vectors generating the extreme rays of $C$.

\begin{theorem}
\label{thm:polytopes}
Let $K,L \in \CK^n_0$ be polytopes and $p \in [-\infty,1]$.
Then $\upmean{p}{K}{L}$ and $\lowmean{-p}{K}{L}$ are also polytopes.
More precisely,
\[
    \upmean{p}{K}{L}
    = \left\{ x \in \R^n :
    a^T x \leq m_p(h_K(a),h_L(a)) \text{ for all }
    a \in \bigcup_{N_K \in \mathcal{N}_K, N_L \in \mathcal{N}_L} \ray(N_K \cap N_L) \right\}
\]
and
\[
    \lowmean{-p}{K}{L}
    = \conv \left( \left\{
    m_{-p}(\gauge{v}_K^{-1},\gauge{v}_L^{-1}) v :
    v \in \bigcup_{F_K \in \mathcal{F}_K, F_L \in \mathcal{F}_L}
    \ray(F_K \cap F_L) \right\} \right).
\]
\end{theorem}

Let us close the introduction with an outline of the paper.
In Section~\ref{sec:prelim}, we introduce the rest of the notation and collect some basic properties of $p$-means used throughout.
We turn to $p$-means of polytopes and Theorem~\ref{thm:polytopes} in Section~\ref{sec:polytopes}.
Afterward, in Section~\ref{sec:boundary}, we study the common boundary structures with regard to the chain \eqref{eq:forward_double_chain},
which allows us to extend results on the optimality of parts of \eqref{eq:forward_double_chain} from \cite{BrDiMe}.
It is also in preparation for the proof of Theorem~\ref{thm:same_p} in Section~\ref{sec:equality_cases},
where we additionally collect results on the equality cases between general $p$-means of convex bodies.
We verify Theorems~\ref{thm:lower_bound}~and~\ref{thm:upper_bounds} in Section~\ref{sec:radii},
based on which we can establish Theorem~\ref{thm:stability}.
Lastly, in Section~\ref{sec:distance}, we show additional applications of Theorems~\ref{thm:lower_bound}~and~\ref{thm:upper_bounds} to the theory of Banach--Mazur distances.
In particular, we show that any Minkowski centered body $K$ is equidistant from all its $p$-symmetrizations, i.e., $p$-means of $K$ and $-K$.

\section{Preliminaries}
\label{sec:prelim}

We begin this section with an introduction of the remaining notation used throughout the paper.
For a number $m \in \N$, we define $[m] \coloneqq \{ 1, \ldots, m\}$.
The \cemph{standard unit vectors} in $\R^n$ are denoted by $u^1, \ldots, u^n$.
The closed \cemph{segment} connecting $x,y \in \R^n$ is given by $[x,y]$.
For a set $X \subset \R^n$, we write $\inte(X)$, $\bd(X)$, and $\cl(X)$ for its \cemph{interior}, \cemph{boundary}, and \cemph{closure}, respectively.

Given $a \in \R^n \setminus \{0\}$ and $\beta \in \R$,
the \cemph{halfspace} $H_{(a,\beta)}^\leq$ is defined as $\{ x \in \R^n : a^T x \leq \beta \}$ (and we use $H_{(a,\infty)}^\leq \coloneqq \R^n$).
The \cemph{hyperplane} $H_{(a,\beta)}$ is the boundary of $H_{(a,\beta)}^\leq$.
A halfspace $H^\leq \subset \R^n$ \cemph{supports} the set $X \subset \R^n$ at $x \in X$ if $X \subset H^\leq$ and $x \in \bd(H^\leq)$,
and a hyperplane $H$ supports $X$ at $x$ if one of the halfspaces bounded by $H$ does.

For bodies $K,L \in \CK^n$, the \cemph{circumradius} of $K$ w.r.t.~$L$ is defined as
\[
    R(K,L)
    \coloneqq \inf \{ \lambda \geq 0 : K \subset \lambda L + c, c \in \R^n \}
    \in [0,\infty].
\]
Any $c \in \R^n$ with $K \subset R(K,L) L + c$ is a \cemph{circumcenter} of $K$ w.r.t.~$L$.
We say that $K$ is \cemph{optimally contained} in $L$, abbreviated by $K \subset^{opt} L$, if $K \subset L$ and $R(K,L) = 1$.
The \cemph{Minkowski asymmetry} of $K$ is defined as $s(K) \coloneqq R(K,-K)$,
and any $c \in \R^n$ with $K - c \subset -s(K) (K - c)$ is a \cemph{Minkowski center} of $K$.
If $c = 0$ is a Minkowski center of $K$, then $K$ is \cemph{Minkowski centered}.
The body $K$ is \cemph{$c$-symmetric} with \cemph{center of symmetry} $c \in \R^n$ if 
$K - c = - (K - c)$,
and \cemph{symmetric} if it has a center of symmetry, i.e., if $s(K)=1$.
For some properties of these notions, see, e.g., \cite{BrDiMe}.

For $K \in \CK^n$, the \cemph{dimension} $\dim(K)$ is the dimension of $\aff(K)$.
We say that $K$ is \cemph{fulldimensional} if $\dim(K) = n$, or equivalently if $\inte(K) \neq \emptyset$.
For fulldimensional $K,L \in \CK^n$, the \cemph{Banach--Mazur distance} between $K$ and $L$ is
\[
    d_{BM}(K,L)
    \coloneqq \inf \{ \lambda \geq 0 : K + c_1 \subset A L + c_2 \subset \lambda (K + c_1), A \in \R^{n \times n} \text{ invertible}, c_1, c_2 \in \R^n \}.
\]
It can be shown (see, e.g., \cite[Lemma~2.1]{XiLe}) that
\[
    d_{BM}(K,L) = \min \{ R(K,AL) R(AL,K) : A \in \R^{n \times n} \text{ invertible} \}.
\]

Next, we collect some basic properties of support and gauge functions used throughout the paper. 
For proofs of almost all of these properties, see \cite[Section~1.7]{Sch}.
The only exception, $h_{AK}(v) = h_K(A^Tv)$ in (iv), follows from simple linear algebra.

\begin{proposition}
\label{prop:support_gauge_fct}
Let $K,L \subset \R^n$ with $0 \in K \cap L$ be closed and convex, $v,w \in \R^{n}$, $\lambda \geq 0$, and $A \in \R^{n \times n}$.
Then:
\begin{enumerate}[(i)]
\item $\gauge{\cdot}_{K} = h_{K^{\circ}}$ and $h_K = \gauge{\cdot}_{K^\circ}$,
\item $h_{K}(\lambda v) = \lambda h_{K}(v)$, $h_{A K}(v) = h_K(A^T v)$, and $h_{-K}(v) = h_K(-v)$,
\item $h_{K}(v+w) \leq h_{K}(v) + h_{K}(w)$,
\item $h_{K}$ is convex, and thus continuous if $K$ is bounded, and
\item $h_{\lambda K + L} = \lambda h_{K} + h_{L}$, and $h_K \leq h_L$ if and only if $K \subset L$.
\end{enumerate}
\end{proposition}

Part (i) gives us results analogous to parts (ii)~--~(v) for $\gauge{\cdot}_K$,
so we do not state them explicitly.
In particular, both $h_K$ and $\gauge{\cdot}_K$ characterize a closed convex set $K \subset \R^n$ with $0 \in K$ in that
if $L \subset \R^n$ is another such set,
then its support/gauge function coincides with that of $K$ if and only if $K = L$.

Like for support and gauge functions, we also collect some basic properties of the $p$-mean $m_p$.

\begin{remark}
\label{rem:mp}
Let $p \in [-\infty,\infty]$ and $\mpvara, \mpvarb > 0$.
It is well-known
that $m_p(\mpvara,\mpvarb)$ is increasing and continuous in all three arguments.
Furthermore, for $p < q$, it holds $m_p(\mpvara,\mpvarb) = m_q(\mpvara,\mpvarb)$ if and only if $\mpvara = \mpvarb$.
If $p \neq \pm \infty$ and $\mpvara \leq \mpvara'$, $\mpvarb \leq \mpvarb'$,
then equality holds in $m_p(\mpvara,\mpvarb) \leq m_p(\mpvara',\mpvarb')$ if and only if $\mpvara = \mpvara'$ and $\mpvarb = \mpvarb'$.
Furthermore, it is easy to see that $m_p(\mpvara,\mpvarb) = 1/m_{-p}(1/\mpvara,1/\mpvarb)$
and for $\gamma > 0$ also $m_p(\gamma \mpvara, \gamma \mpvarb) = \gamma m_p(\mpvara, \mpvarb)$.
The well-known (reverse) Minkowski inequality states that $m_p$ is convex when $p \in [1,\infty]$ (and concave when $p \in [-\infty,1]$).

The above properties of $m_p$ show for $K,L \in \CK^n_0$, $p \in [-\infty,\infty]$, and $v \in \R^n \setminus \{0\}$ that
\[
    m_{-p}(\|v\|_K^{-1}, \|v\|_L^{-1}) v
    = \frac{v}{m_p(\|v\|_K, \|v\|_L)}.
\]
We use this identity frequently throughout depending on which form is more convenient.
\end{remark}

We end this section with a lemma collecting some basic properties of upper and lower $p$-means of convex bodies.
Part (ii) already verifies the set containment in Theorem~\ref{thm:same_p}.
The equality case requires a more in-depth analysis,
which is provided in Sections~\ref{sec:boundary}~and~\ref{sec:equality_cases}.

\begin{lemma}
\label{lem:p-mean_basics}
Let $K, K', L, L' \in \CK^n_0$ with $K \subset K'$, $L \subset L'$, $p, q \in [-\infty,\infty]$ with $p \leq q$, and $A \in \R^{n \times n}$ invertible.
Then:
\begin{enumerate}[(i)]
\item $(\upmean{p}{K}{L})^\circ = \lowmean{-p}{K^\circ}{L^\circ}$ and $(\lowmean{p}{K}{L})^\circ = \upmean{-p}{K^\circ}{L^\circ}$,
\item $\lowmean{p}{K}{L} \subset \upmean{p}{K}{L}$,
\item 	$\upmean{p}{K}{L} \subset \upmean{q}{K'}{L'}$ and $\lowmean{p}{K}{L} \subset \lowmean{q}{K'}{L'}$, and
\item $\upmean{p}{A K}{A L} = A \cdot \upmean{p}{K}{L}$ and $\lowmean{p}{A K}{A L} = A \cdot \lowmean{p}{K}{L}$.
\end{enumerate}
\end{lemma}

Note that the invertibility assumption for $A$ is necessary. For example, it is not difficult to see that if $A$ is an orthogonal projection, the equality $A (K \cap L) = (A K) \cap (A L)$ may fail.

\begin{proof}
For (i), we observe that
\begin{align*}
	\left( \upmean{p}{K}{L} \right)^\circ
		& = \left( \bigcap_{a \in \mS_2} H_{(a,m_p(h_K(a),h_L(a)))}^\leq \right)^\circ
		= \cl \left( \conv \left( \bigcup_{a \in \mS_2} \left( H_{(a,m_p(h_K(a),h_L(a)))}^\leq \right)^\circ \right) \right) \\
		& = \cl \left( \conv \left( \left\{ \frac{a}{m_p(h_K(a),h_L(a))} : a \in \mS_2 \right\} \right) \right).
\end{align*}
The set that we take the convex hull and closure
of in the last step is compact by compactness of $\mS_2$ and continuity of $m_p(h_K,h_L)$.
Thus, we may drop the closure here.
Since $h_K = \gauge{\cdot}_{K^\circ}$ and $h_L = \gauge{\cdot}_{L^\circ}$ by Proposition~\ref{prop:support_gauge_fct}~(i),
we conclude that $(\upmean{p}{K}{L})^\circ = \lowmean{-p}{K^\circ}{L^\circ}$.
Now, (i) follows from
$(\lowmean{p}{K}{L})^\circ = (\lowmean{-(-p)}{(K^\circ)^\circ}{(L^\circ)^\circ})^\circ
= ((\upmean{-p}{K^\circ}{L^\circ})^\circ)^\circ = \upmean{-p}{K^\circ}{L^\circ}$.

For (ii), let $x, a \in \mS_2$.
Then $a^T x \leq 0$, or $a^T x > 0$ and
\[
	a^T \frac{x}{m_{-p}(\gauge{x}_K,\gauge{x}_L)}
    = \frac{1}{m_{-p} \left( \frac{\gauge{x}_K}{a^T x},\frac{\gauge{x}_L}{a^T x} \right)}
	= m_p \left( a^T \normal{x}{K},a^T \normal{x}{L} \right)
	\leq m_p(h_K(a),h_L(a)).
\]
Since $a \in \mS_2$ has been chosen arbitrarily, we obtain
$x/m_{-p}(\gauge{x}_K,\gauge{x}_L) \in \upmean{p}{K}{L}$.
Because $x \in \mS_2$ has also been chosen arbitrarily and $\upmean{p}{K}{L}$ is clearly convex, (ii) is proven.

(iii) is an immediate consequence of the monotonicity of $p$-means of real numbers.

Finally, for (iv), let $x \in \R^n$.
Then $x$ lies in $\upmean{p}{A K}{A L}$ if and only if $a^T x \leq m_p(h_{A K}(a),h_{A L}(a))$ for all $a \in \R^n$.
This is equivalent to $(A^T a)^T (A^{-1} x) \leq m_p(h_K(A^T a),h_L(A^T a))$ for all $a \in \R^n$,
i.e., $A^{-1} x \in \upmean{p}{K}{L}$.
Consequently, $\upmean{p}{A K}{A L} = A \cdot \upmean{p}{K}{L}$.
Since
$(A C)^\circ = A^{-T} C^\circ$ for all $C \in \CK^n_0$,
we also obtain the second equality in (iv) from (i).
\end{proof}

\section{\texorpdfstring{$p$}{p}-Means of Polytopes}
\label{sec:polytopes}

The goal of this section is to show that certain $p$-means of polytopes are guaranteed to again be polytopes,
that is, we verify Theorem~\ref{thm:polytopes}.
While the polytopal representations provided in the theorem may be irreducible in general,
it is possible to reduce them further in selected cases.
We provide such an example at the end of the section,
which is important for the analysis of some equality cases between $p$-means in Section~\ref{sec:equality_cases}.

\begin{proof}[Proof of Thereom~\ref{thm:polytopes}]
We only need to show that $\lowmean{-p}{K}{L}$ is a polytope
with the claimed representation
for polytopes $K,L \in \CK^n_0$ and $p \in [-\infty,1]$.
Under these conditions, $K^\circ$ and $L^\circ$ are also polytopes
and therefore, by Lemma~\ref{lem:p-mean_basics}~(i),
$\upmean{p}{K}{L} = \left( \lowmean{-p}{K^\circ}{L^\circ} \right)^\circ$
is a polytope with the claimed representation for upper $p$-means as well.

Let $a,b \in \R^n \setminus \{0\}$ such that $K \cap H_{(a,1)}$ and $L \cap H_{(b,1)}$ are facets of the respective polytopes
and write $F_K \coloneqq \pos(K \cap H_{(a,1)})$ and $F_L \coloneqq \pos(L \cap H_{(b,1)})$.
It is easy to see for $v \in F_K$ that
\begin{equation}
\label{eq:local_polytope_gauge}
	\gauge{v}_K
    = a^T v.
\end{equation}
Now, let $v,w \in (F_K \cap F_L) \setminus \{ 0 \}$.
We define $\bar{v} \coloneqq m_{-p}(\gauge{v}_K^{-1}, \gauge{v}_L^{-1}) v$ and $\bar{w} \coloneqq m_{-p}(\gauge{w}_K^{-1}, \gauge{w}_L^{-1}) w$.
For $\lambda \in [0,1]$, let $x \coloneqq \lambda \bar{v} + (1 - \lambda) \bar{w} \in \lowmean{-p}{K}{L}$.
We claim that
\[
	m_{-p}(\gauge{x}_K^{-1}, \gauge{x}_L^{-1}) x \in [0, x],
\]
which is equivalent to
\begin{equation}
\label{eq:polytope_required_ineq}
	1 \leq m_p(\gauge{\lambda \bar{v} + (1 - \lambda) \bar{w}}_K, \gauge{\lambda \bar{v} + (1 - \lambda) \bar{w}}_L)
\end{equation}
by Remark~\ref{rem:mp}.
Using \eqref{eq:local_polytope_gauge} and again Remark~\ref{rem:mp}, the right-hand side can be rewritten as
\[
	m_p \left( \frac{\lambda a^T v}{m_p(a^T v, b^T v)} + \frac{(1 - \lambda) a^T w}{m_p(a^T w, b^T w)},
		\frac{\lambda b^T v}{m_p(a^T v, b^T v)} + \frac{(1 - \lambda) b^T w}{m_p(a^T w, b^T w)} \right).
\]
The concavity of $m_p$ for $p \in [-\infty,1]$ (see Remark~\ref{rem:mp})
shows that this term is lower bounded by
\begin{align*}
    & \lambda m_p \left( \frac{a^T v}{m_p(a^T v, b^T v)}, \frac{b^T v}{m_p(a^T v, b^T v)} \right)
    + (1-\lambda) m_p \left( \frac{a^T w}{m_p(a^T w, b^T w)}, \frac{b^T w}{m_p(a^T w, b^T w)} \right)
    \\
    & = \lambda \frac{m_p(a^T v,b^T v)}{m_p(a^T v, b^T v)}
    + (1-\lambda) \frac{m_p(a^T w,b^T w)}{m_p(a^T w, b^T w)}
    = 1,
\end{align*}
which yields the claimed inequality \eqref{eq:polytope_required_ineq}.

The above shows that
\begin{align*}
	& \conv \left( \left\{ 0 \right\} \cup \left\{ m_{-p}(\gauge{v}_K^{-1}, \gauge{v}_L^{-1}) v :
			v \in \mS_2 \cap F_K \cap F_L \right\} \right) \\
	& = \conv \left( \left\{ 0 \right\} \cup \left\{ m_{-p}(\gauge{v}_K^{-1}, \gauge{v}_L^{-1}) v : v \in \ray(F_K \cap F_L) \right\} \right).
\end{align*}
Any $v \in \mS_2$ lies in $F_K \cap F_L$ like above for some facets of $K$ and $L$,
so the claimed representation of $\lowmean{-p}{K}{L}$ follows.
Since $\mathcal{F}_K$ and $\mathcal{F}_L$ are finite,
and each of the sets $\ray(F_K \cap F_L)$ is finite as the set of extreme rays of a polyhedral cone,
$\lowmean{-p}{K}{L}$ is a polytope.
\end{proof}

For $n = 2$ and $p \in [-\infty,1]$,
Theorem~\ref{thm:polytopes} shows that $\lowmean{-p}{K}{L}$
has a particularly simple structure if $K$ and $L$ are polytopes.
First, note that $\normal{v}{2}$ lies in $\bigcup_{F_K \in \mathcal{F}_K, F_L \in \mathcal{F}_L} \ray(F_K \cap F_L)$ for any $v \in \ext(K)$ (even if $n \geq 3$):
Simply take any $F_K \in \mathcal{F}_K$ and $F_L \in \mathcal{F}_L$ that contain $v$.
Then $[0,\infty) v$ is an extreme ray of $F_K$ and therefore also of $F_K \cap F_L$.
The same argument can be used for any $v \in \ext(L)$.
Conversely, for $n = 2$ it is easy to see that the sets $\ray(F_K \cap F_L)$ consist of at most two elements of $\mS_2 \cap \pos(\ext(K) \cup \ext(L))$.
In summary, we have that the vertices of $\lowmean{-p}{K}{L}$ are simply dilates
of the vertices of $K$ and $L$ in this case.
Using $\upmean{p}{K}{L} = (\lowmean{-p}{K^{\circ}}{L^{\circ}})^{\circ}$,
the facet outer normals of $\upmean{p}{K}{L}$ are facet outer normals of $K$ or $L$ under the same conditions.

In the following, we present an example where the representations obtained in Theorem~\ref{thm:polytopes} can be reduced further.

\begin{example}
\label{ex:polytope_redundant}
Let $p \in (-1,\infty]$.
Then $m_p(2,\frac{2}{3}) > m_{-1}(2,\frac{2}{3}) = 1$.
By continuity, there exists $\gamma \in (1,2)$ close to $2$ with $m_p(\gamma,\frac{\gamma}{\gamma+1}) > 1$.
For $n \geq 2$, we define
\[
	K \coloneqq \conv( \{u^n, \pm u^i, \pm \gamma u^i - u^n : i \in [n-1] \} ) \in \CK^n_0
\]
(cf.~Figure~\ref{fig:polytope_redundant}).

\begin{figure}[ht]
\def\scale{2}
\begin{tikzpicture}[scale=\scale, line join=round]
\def\g{(16/9)}

\draw[dotted] (0,0) -- ({\g},1);
\draw[dotted] (0,0) -- ({-\g},1);

\draw[blue,thick] ({\g},-1) -- ({\g/(\g+1)},{1/(\g+1)});
\draw[blue,thick] ({-\g},-1) -- ({-\g/(\g+1)},{1/(\g+1)});

\draw[thick] ({-\g},-1) -- ({\g},-1) -- (1,0) -- (0,1) -- (-1,0) -- ({-\g},-1);

\fill ({-\g},-1) circle [radius=\ds];
\fill ({\g},-1) circle [radius=\ds];
\fill (-1,0) circle [radius=\ds];
\fill (1,0) circle [radius=\ds];
\fill (0,1) circle [radius=\ds];
\fill[blue] ({\g/(\g+1)},{1/(\g+1)}) circle [radius=\ds];
\fill[blue] ({-\g/(\g+1)},{1/(\g+1)}) circle [radius=\ds];

\draw[dashed,thick] ({-\g},1) -- ({\g},1);
\draw[dashed,thick] ({\g},1) -- (1,0) -- (0,-1);
\draw[dashed,thick] (0,-1) -- (-1,0) -- ({-\g},1);

\draw[orange,thick] (0,-1) -- ({\g/sqrt(\g+1)},{-1/sqrt(\g+1)}) -- ({\g/sqrt(\g+1)},{1/sqrt(\g+1)}) -- (0,1)
				-- ({-\g/sqrt(\g+1)},{1/sqrt(\g+1)}) -- ({-\g/sqrt(\g+1)},{-1/sqrt(\g+1)}) -- cycle;

\fill (0,0) circle [radius=\ds] node[anchor=north west] {$0$};
\end{tikzpicture}
\caption{
The bodies considered in Example~\ref{ex:polytope_redundant} for $n = 2$, $p = 0$, and $\gamma = \frac{16}{9}$:
$K$ (black), $-K$ (dashed), $K'$ (blue), $\lowmean{p}{K}{-K}$ (orange).
}
\label{fig:polytope_redundant}
\end{figure}

Applying Theorem~\ref{thm:polytopes}, one obtains (with some calculations) that
\begin{equation}
\label{eq:ex_polytope_vrep}
	\lowmean{p}{K}{-K}
	= \conv \left( \left\{ \pm u^n, \pm u^i, \frac{\pm \gamma u^i \pm u^n}{m_{-p}(\gauge{\gamma u^1 + u^n}_K,1)} : i \in [n-1] \right\} \right).
\end{equation}
By $\frac{\gamma}{\gamma+1} u^1 + \frac{1}{\gamma+1} u^n \in [u^1,u^n] \subset K$, it holds $\gauge{\gamma u^1 + u^n}_K \leq \gamma + 1$.
Additionally, $K \subset H_{(u^1+u^n,1)}^\leq$ yields $\gauge{\gamma u^1 + u^n}_K \geq (u^1 + u^n)^T (\gamma u^1 + u^n) = \gamma +1$.
Altogether, with Remark~\ref{rem:mp}, we can simplify
\[
	\frac{\pm \gamma u^i \pm u^n}{m_{-p}(\gauge{\gamma u^1 + u^n}_K,1)}
		= \frac{\pm \gamma u^i \pm u^n}{m_{-p}(\gamma + 1,1)}
		= \pm m_p \left( \gamma,\frac{\gamma}{\gamma+1} \right) u^i \pm m_p \left( 1,\frac{1}{\gamma + 1} \right) u^n.
\]
By $m_p(\gamma,\frac{\gamma}{\gamma+1}) > 1$, we obtain $\pm u^i \in \inte(\lowmean{p}{K}{-K})$.
Thus, the example shows that the representation of $p$-means obtained
from Theorem~\ref{thm:polytopes} may be redundant,
even if we choose $L = -K$.

We discuss some additional properties of this example, which are used later on. With
\[
	K' \coloneqq \conv \left( \left\{ u^n, \pm \gamma u^i - u^n, \frac{\pm \gamma u^i + u^n}{\gamma + 1} :
					i \in [n-1] \right\} \right)
	\subsetneq K
\]
it holds $\lowmean{p}{K'}{-K'} \subset \lowmean{p}{K}{-K}$ by Lemma~\ref{lem:p-mean_basics}~(iii),
and $\lowmean{p}{K}{-K} \subset \lowmean{p}{K'}{-K'}$
since all points in the right-hand set in \eqref{eq:ex_polytope_vrep} are by definition contained in $\lowmean{p}{K'}{-K'}$.
In summary, we have $\lowmean{p}{K'}{-K'} = \lowmean{p}{K}{-K}$.
This shows that $\underline{M}_p$ is not strictly increasing in its two set-arguments for $p \in (-1,\infty]$ and $n \geq 2$.
Considering the polars of the bodies constructed above yields by Lemma~\ref{lem:p-mean_basics}~(i) that $\overline{M}_p$ is also not strictly increasing for $p \in [-\infty,1)$ and $n \geq 2$.

In contrast, we show $\upmean{p}{K'}{-K'} \subsetneq \upmean{p}{K}{-K}$ for $p \in [-\infty,1)$ (and, again using Lem\-ma~\ref{lem:p-mean_basics}~(i), also
$\lowmean{p}{K^\circ}{-K^\circ} \subsetneq \lowmean{p}{(K')^\circ}{-(K')^\circ}$ when $p \in (-1,\infty]$).
We have
\begin{align*}
    h_{\upmean{p}{K'}{-K'}}((\gamma+2) u^1 + \gamma^2 u^n)
    & \leq m_p(h_{K'}((\gamma+2) u^1 + \gamma^2 u^n),h_{-K'}((\gamma+2) u^1 + \gamma^2 u^n))
    \\
    & = m_p(2 \gamma, 2 \gamma (\gamma + 1)),
\end{align*}
so $\lambda u^1 \notin \upmean{p}{K'}{-K'}$ for any $\lambda > \frac{2 \gamma}{\gamma + 2} m_p(1, \gamma+1)$.
Theorem~\ref{thm:polytopes} can be used to show that all facets of $\upmean{p}{K}{-K}$ are parallel to facets of $K$ or $-K$,
which yields $m_p(1,2\gamma-1) u^1 \in \upmean{p}{K}{-K}$ after some computation.
Finally,
\[
    m_p \left( \frac{1}{\gamma+2}, \frac{\gamma+1}{\gamma+2} \right)
    < m_p \left( \frac{1}{2\gamma}, \frac{2\gamma - 1}{2\gamma} \right)
\]
from the concavity of $m_p$ and $\frac{1}{\gamma+2} < \frac{1}{2\gamma} < \frac{1}{2}$ verifies $m_p(1, 2\gamma - 1) u^1 \notin \upmean{p}{K'}{-K'}$. Hence,
$\upmean{p}{K'}{-K'} \subsetneq \upmean{p}{K}{-K}$ is true as claimed.
\end{example}

\section{Common Boundary Structures}
\label{sec:boundary}

Our goal in this section is to provide some technical details on the common boundary points and common supporting hyperplanes of the $p$-means in the connected double chain \eqref{eq:forward_double_chain}.
These insights are important for our proof of Theorem~\ref{thm:same_p} in Section~\ref{sec:equality_cases}, as well as some details about equality cases in Section~\ref{sec:radii}.
Moreover, they allow us to generalize a result about the optimality of the containment between the standard means in \eqref{eq:forward_double_chain} presented in \cite{BrDiMe} to the full double chain, while also revealing some new aspects of the optimal containment within the individual chains of the upper and lower $p$-means.
These results about optimal containment are presented at the end of this section.

The following two technical lemmas address the question of when
the unique boundary points of two $p$-means on a ray $[0,\infty) x$ for $x \in \R^n \setminus \{ 0 \}$ coincide.
This is equivalent to the equality of the respective gauge function values.
By Proposition~\ref{prop:support_gauge_fct}~(i)
and Lemma~\ref{lem:p-mean_basics}~(i), this also yields results on the equality of support functions, which we do not state explicitly.

We begin by comparing upper and lower $p$-means for the same finite $p$-value.
The cases for $p = \pm \infty$ are covered by Theorem~\ref{thm:same_p},
which we prove in the next section.
The containments from right to left and the implication from (iii) to (i) and (ii) in the following lemma
are for $p = 0$ also shown in \cite[Proposition~15]{MiRog1}.
Note that while the second set is clearly contained in the first one, the reverse inclusion is non-trivial for $p < 1$
since $h_{\upmean{p}{K}{L}}(a) < m_p(h_K(a),h_L(a))$ is possible for general $a \in \R^n$ in that case.

\begin{lemma}
\label{lem:common_boundary_same_p}
Let $K,L \in \CK^n_0$, $p \in \R$, and $x \in \R^n \setminus \{ 0 \}$.
Then
\begin{align*}
	N_{\upmean{p}{K}{L}} \left( \normal{x}{\lowmean{p}{K}{L}} \right)
		& = \left\{ a \in \R^n : a^T \normal{x}{\lowmean{p}{K}{L}} = m_p(h_K(a),h_L(a)) \right\} \\
		& = N_K \left( \normal{x}{K} \right) \cap N_L \left( \normal{x}{L} \right).
\end{align*}
Moreover, the following are equivalent:
\begin{enumerate}[(i)]
\item $\gauge{x}_{\lowmean{p}{K}{L}} = \gauge{x}_{\upmean{p}{K}{L}}$.
\item $\gauge{x}_{\upmean{p}{K}{L}} = m_{-p}(\gauge{x}_K,\gauge{x}_L)$.
\item  $N_K \left( \normal{x}{K} \right) \cap N_L \left( \normal{x}{L} \right) \supsetneq \{ 0 \}$.
\end{enumerate}
\end{lemma} \begin{proof}
By Lemma~\ref{lem:p-mean_basics}~(ii) and the definition of $\lowmean{p}{K}{L}$, we have
\[
	\gauge{x}_{\upmean{p}{K}{L}} \leq \gauge{x}_{\lowmean{p}{K}{L}} \leq m_{-p}(\gauge{x}_K,\gauge{x}_L).
\]
This already yields the implication from (ii) to (i).
The definition of $\upmean{p}{K}{L}$ now implies for any $a \in \R^n$ that $a^T x \leq 0 < m_p(h_K(a),h_L(a))$, or $a^T x > 0$ and
\begin{equation}
\label{eq:inner_product_chain}
	a^T \frac{x}{m_{-p}(\gauge{x}_K,\gauge{x}_L)}
	\leq a^T \normal{x}{\lowmean{p}{K}{L}}
	\leq a^T \normal{x}{\upmean{p}{K}{L}}
	\leq h_{\upmean{p}{K}{L}}(a)
	\leq m_p(h_K(a),h_L(a)).
\end{equation}
If $a \in N_K \left( \normal{x}{K} \right) \cap N_L \left( \normal{x}{L} \right)$, we additionally obtain
\[
	m_p(h_K(a),h_L(a))
		= m_p \left( a^T \normal{x}{K}, a^T \normal{x}{L} \right)
		= a^T \frac{x}{m_{-p}(\gauge{x}_K,\gauge{x}_L)},
\]
and therefore equality in the full chain in \eqref{eq:inner_product_chain} as well.
Hence, the implication from (iii) to (ii) and the set inclusions from right to left in the lemma follow.

Next, we assume $\gauge{x}_{\lowmean{p}{K}{L}} = \gauge{x}_{\upmean{p}{K}{L}}$.
By the definition of $\lowmean{p}{K}{L}$,
there exist some $v^1, \ldots, v^k \in \mS_2$ and $\lambda_1, \ldots, \lambda_k > 0$ such that
$\sum_{i \in [k]} \lambda_i = 1$ and
\begin{equation}
\label{eq:common_boundary_y}
	y \coloneqq \normal{x}{\lowmean{p}{K}{L}}
		= \sum_{i \in [k]} \lambda_i  \frac{v^i}{m_{-p}(\gauge{v^i}_K, \gauge{v^i}_L)}.
\end{equation}
It follows directly from the definition of $\upmean{p}{K}{L}$ that
\[
    \sup \left\{ \frac{a^T x}{m_p(h_K(a), h_L(a))} : a \in \mS_2 \right\}
    = \gauge{x}_{\upmean{p}{K}{L}},
\]
where the supremum is attained for some $\bar{a} \in \mS_2$ by continuity in $a$.
Therefore, by the assumption $\gauge{x}_{\lowmean{p}{K}{L}} = \gauge{x}_{\upmean{p}{K}{L}}$,
there exists $\bar{a} \in \mS_2$ with $\bar{a}^T y = m_p(h_K(\bar{a}),h_L(\bar{a}))$.
Since for all $i \in [k]$
\[
	y^i
		\coloneqq \frac{v^i}{m_{-p}(\gauge{v^i}_K, \gauge{v^i}_L)}
		\in \lowmean{p}{K}{L}
		\subset \upmean{p}{K}{L}
		\subset H_{(\bar{a},m_p(h_K(\bar{a}),h_L(\bar{a})))}^\leq
		= H_{(\bar{a},\bar{a}^T y)}^\leq,
\]
all $y^i$ must, because of \eqref{eq:common_boundary_y}, even belong to $H_{(\bar{a},\bar{a}^T y)}$.
Hence,
\[
	m_p(h_K(\bar{a}), h_L(\bar{a}))
		= \bar{a}^T y
		= \bar{a}^T \frac{v^i}{m_{-p}(\gauge{v^i}_K, \gauge{v^i}_L)}
		= m_p \left( \bar{a}^T \frac{v^i}{\gauge{v^i}_K}, \bar{a}^T \frac{v^i}{\gauge{v^i}_L} \right).
\]
Since $p \neq \pm \infty$, we obtain $\bar{a}^T \normal{v^i}{K} = h_K(\bar{a})$ and $\bar{a}^T \normal{v^i}{L} = h_L(\bar{a})$.
With $\mu_i \coloneqq \frac{\lambda_i \gauge{v^i}_K}{m_{-p}(\gauge{v^i}_K,\gauge{v^i}_L)} > 0$, this yields
\[
	\frac{x}{\sum_{j \in [k]} \mu_j \gauge{x}_{\lowmean{p}{K}{L}}}
	= \sum_{i \in [k]} \frac{\mu_i}{\sum_{j \in [k]} \mu_j} \cdot \normal{v^i}{K}
	\in \conv \left( \left\{ \normal{v^1}{K}, \ldots, \normal{v^k}{K} \right\} \right)
	\subset K \cap H_{(\bar{a},h_K(\bar{a}))}.
\]
It follows $\sum_{j \in [k]} \mu_j \gauge{x}_{\lowmean{p}{K}{L}} = \gauge{x}_K$ and $\bar{a} \in N_K \left( \normal{x}{K} \right)$.
We can argue analogously for $\bar{a} \in N_L \left( \normal{x}{L} \right)$.
This shows the implication from (i) to (iii) and that the second set in the lemma is contained in the third one,
as any $a \in \R^n \setminus \{ 0 \}$ with $a^T \normal{x}{\lowmean{p}{K}{L}} = m_p(h_K(a),h_L(a))$ can be rescaled to lie in $\mS_2$ and thus be a valid choice for $\bar{a}$.

Finally, assume $a \in N_{\upmean{p}{K}{L}} \left( \normal{x}{\lowmean{p}{K}{L}} \right)$.
Then, with Proposition~\ref{prop:support_gauge_fct}~(i) and Lemma~\ref{lem:p-mean_basics}~(i),
\[
	\gauge{a}_{\lowmean{-p}{K^\circ}{L^\circ}} = h_{\upmean{p}{K}{L}}(a) = h_{\lowmean{p}{K}{L}}(a) = \gauge{a}_{\upmean{-p}{K^\circ}{L^\circ}},
\]
which by the already proven equivalence of (i) and (ii) yields
\[
	a^T \normal{x}{\lowmean{p}{K}{L}}
		= h_{\upmean{p}{K}{L}}(a)
		= \gauge{a}_{\upmean{-p}{K^\circ}{L^\circ}}
		= m_p(\gauge{a}_{K^\circ},\gauge{a}_{L^\circ})
		= m_p(h_K(a),h_L(a)).
\]
Thus, the first set in the lemma is contained in the second one, which completes the proof.
\end{proof}

Next, we compare $p$-means for different $p$-values.

\begin{lemma}
\label{lem:common_boundary_different_p}
Let $K,L \in \CK^n_0$, $p, q \in [-\infty,\infty]$ with $p < q$, and $x \in \R^n \setminus \{ 0 \}$.
Then
\begin{enumerate}[(i)]
\item $\gauge{x}_{\lowmean{p}{K}{L}} = \gauge{x}_{\lowmean{q}{K}{L}}$ if and only if
		$\gauge{x}_K = \gauge{x}_L = \gauge{x}_{\lowmean{q}{K}{L}}$.
\item $\gauge{x}_{\upmean{p}{K}{L}} = \gauge{x}_{\upmean{q}{K}{L}}$ if and only if
		$N_{\upmean{q}{K}{L}} \left( \normal{x}{\upmean{p}{K}{L}} \right) \supsetneq \{ 0 \}$, where
\begin{align*}
	N_{\upmean{q}{K}{L}} \left( \normal{x}{\upmean{p}{K}{L}} \right)
					&= \left\{ a \in \R^n : a^T \normal{x}{\upmean{p}{K}{L}} = m_q(h_K(a),h_L(a)) \right\} \\
					&= \left\{ a \in \R^n : a^T \normal{x}{\upmean{p}{K}{L}} = h_K(a) = h_L(a) \right\}.
\end{align*}
\item $\gauge{x}_{\lowmean{p}{K}{L}} = \gauge{x}_{\upmean{q}{K}{L}}$ if and only if $\gauge{x}_K = \gauge{x}_L = \gauge{x}_{\upmean{q}{K}{L}}$ and
\[
	N_{\upmean{q}{K}{L}} \left( \normal{x}{\lowmean{p}{K}{L}} \right)
		= N_{\upmean{\infty}{K}{L}} \left( \normal{x}{\lowmean{p}{K}{L}} \right)
		= N_K \left( \normal{x}{K} \right) \cap N_L \left( \normal{x}{L} \right)
		\supsetneq \{ 0 \}.
\]
\end{enumerate}
\end{lemma}

Let us point out that while for $q \in [-\infty,-1]$ the equality $\gauge{x}_K = \gauge{x}_L$
always also implies $\gauge{x}_K = \gauge{x}_L = \gauge{x}_{\lowmean{q}{K}{L}}$,
this is not necessarily the case for $q > -1$ (cf.~Figure~\ref{fig:common_boundary}).
Similarly, the first set equality in (ii) is again non-trivial for $q < 1$.
Furthermore, the second set equality in (ii) does not imply that $\normal{x}{\upmean{p}{K}{L}}$ lies in $K$ or $L$
(cf.~Figure~\ref{fig:common_boundary}).

\begin{figure}[ht]
\def\scale{1.25}
\begin{tikzpicture}[scale=\scale, line join=round]
\def\p{3}

\draw[thick] ({-sqrt(3)},-1) -- ({sqrt(3)},-1) -- (0,2) -- cycle;

\draw[dashed,thick] ({sqrt(3)},1) -- ({-sqrt(3)},1);
\draw[dashed,thick] ({-sqrt(3)},1) -- (0,-2);
\draw[dashed,thick] (0,-2) -- ({sqrt(3)},1);

\draw[dashdotted,thick] (0,-2) -- ({-sqrt(3)},-1);
\draw[dashdotted,thick] (0,-2) -- ({sqrt(3)},-1);
\draw[dashdotted,thick] ({sqrt(3)},1) -- ({sqrt(3)},-1);
\draw[dashdotted,thick] ({sqrt(3)},1) -- (0,2);
\draw[dashdotted,thick] ({-sqrt(3)},1) -- (0,2);
\draw[dashdotted,thick] ({-sqrt(3)},1) -- ({-sqrt(3)},-1);

\draw[red,thick,variable=\t,domain={1/(2^(\p-1)+1)}:{2^(\p-1)/(2^(\p-1)+1)}] 
		plot	(	{-sqrt(3)*(1-\t) / (2^(1/\p) * (\t^(\p/(\p-1))+(1-\t)^(\p/(\p-1)))^((\p-1)/\p))},
				{(-(1-\t) -2*\t) / (2^(1/\p) * (\t^(\p/(\p-1))+(1-\t)^(\p/(\p-1)))^((\p-1)/\p))}	)
	--	plot	(	{sqrt(3)*\t / (2^(1/\p) * (\t^(\p/(\p-1))+(1-\t)^(\p/(\p-1)))^((\p-1)/\p))},
				{(-2*(1-\t) - \t) / (2^(1/\p) * (\t^(\p/(\p-1))+(1-\t)^(\p/(\p-1)))^((\p-1)/\p))}	)
	--	plot	(	{sqrt(3) / (2^(1/\p) * (\t^(\p/(\p-1))+(1-\t)^(\p/(\p-1)))^((\p-1)/\p))},
				{(2*\t -1) / (2^(1/\p) * (\t^(\p/(\p-1))+(1-\t)^(\p/(\p-1)))^((\p-1)/\p))}	)
	--	plot	(	{sqrt(3)*(1-\t) / (2^(1/\p) * (\t^(\p/(\p-1))+(1-\t)^(\p/(\p-1)))^((\p-1)/\p))},
				{((1-\t) + 2*\t) / (2^(1/\p) * (\t^(\p/(\p-1))+(1-\t)^(\p/(\p-1)))^((\p-1)/\p))}	)
	--	plot	(	{-sqrt(3)*\t / (2^(1/\p) * (\t^(\p/(\p-1))+(1-\t)^(\p/(\p-1)))^((\p-1)/\p))},
				{(2*(1-\t) + \t)) / (2^(1/\p) * (\t^(\p/(\p-1))+(1-\t)^(\p/(\p-1)))^((\p-1)/\p))}	)
	--	plot	(	{-sqrt(3) / (2^(1/\p) * (\t^(\p/(\p-1))+(1-\t)^(\p/(\p-1)))^((\p-1)/\p))},
				{(-\t + (1-\t)) / (2^(1/\p) * (\t^(\p/(\p-1))+(1-\t)^(\p/(\p-1)))^((\p-1)/\p))}	)
	--		cycle;

\draw[blue,thick] ({-sqrt(3)},0) -- ({-sqrt(3)/2},-1.5) -- ({sqrt(3)/2},-1.5) -- ({sqrt(3)},0) -- ({sqrt(3)/2},1.5) -- ({-sqrt(3)/2},1.5) -- cycle;

\draw[orange,thick] ({sqrt(3/2)},{sqrt(1/2)}) -- (0,{sqrt(2)}) -- ({-sqrt(3/2)},{sqrt(1/2)})
			-- ({-sqrt(3/2)},{-sqrt(1/2)}) -- (0,{-sqrt(2)}) -- ({sqrt(3/2)},{-sqrt(1/2)}) -- cycle;

\fill (0,0) circle [radius=\ds] node[anchor=west] {0};
\end{tikzpicture}
\caption{
An example for Lemma~\ref{lem:common_boundary_different_p}:
$K$ (black), $-K$ (dashed), $\lowmean{0}{K}{-K}$ (orange),
$\upmean{1}{K}{-K}$ (blue), $\upmean{3}{K}{-K}$ (red), $\upmean{\infty}{K}{-K}$ (dash-dotted).
The common boundary points of $K$ and $-K$ are not boundary points of $\lowmean{0}{K}{-K}$.
Furthermore, the vertices of $\upmean{1}{K}{-K}$ are smooth boundary points of $\upmean{\infty}{K}{-K}$.
By Lemma~\ref{lem:common_boundary_different_p}~(ii), $\upmean{p}{K}{-K}$ for $p > 1$ is supported at each of these points
by exactly one respective line that also supports $K$ and $-K$.
However, this does not mean that these points must belong to $K$ or $-K$.
}
\label{fig:common_boundary}
\end{figure}

\begin{proof}
For (i), if $\gauge{x}_K = \gauge{x}_L = \gauge{x}_{\lowmean{q}{K}{L}}$,
then Lemma~\ref{lem:p-mean_basics}~(iii) and the definition of $\lowmean{p}{K}{L}$ show
\[
	\gauge{x}_{\lowmean{q}{K}{L}} \leq \gauge{x}_{\lowmean{p}{K}{L}} \leq m_{-p}(\gauge{x}_K,\gauge{x}_L) = \gauge{x}_{\lowmean{q}{K}{L}},
\]
which already yields $\gauge{x}_{\lowmean{p}{K}{L}} = \gauge{x}_{\lowmean{q}{K}{L}}$.
Conversely, if we assume $\gauge{x}_{\lowmean{p}{K}{L}} = \gauge{x}_{\lowmean{q}{K}{L}}$,
then there exist $v^1, \ldots, v^k \in \mS_2$ and $\lambda_1, \ldots, \lambda_k > 0$ such that $\sum_{i \in [k]} \lambda_i = 1$ and
\[
	y \coloneqq
		\normal{x}{\lowmean{q}{K}{L}}
		= \sum_{i \in [k]} \lambda_i  \frac{v^i}{m_{-p}(\gauge{v^i}_K, \gauge{v^i}_L)}.
\]
Since $y \in \bd(\lowmean{q}{K}{L})$, there exists a hyperplane $H$ that supports $\lowmean{q}{K}{L}$ at $y$.
By $\lowmean{p}{K}{L} \subset \lowmean{q}{K}{L}$ from Lemma~\ref{lem:p-mean_basics}~(iii),
all $v^i / m_{-p}(\gauge{v^i}_K, \gauge{v^i}_L)$ must also lie in $H$,
which shows that they are boundary points of $\lowmean{q}{K}{L}$ as well.
Thus, for all $i \in [k]$ we obtain
\[
	m_{-p}(\gauge{v^i}_K, \gauge{v^i}_L)
		= \gauge{v^i}_{\lowmean{q}{K}{L}}
		\leq m_{-q}(\gauge{v^i}_K, \gauge{v^i}_L),
\]
which by $-q < -p$ implies $\gauge{v^i}_K = \gauge{v^i}_L$.
Hence, all
\[
	y^i \coloneqq \frac{v^i}{m_{-p}(\gauge{v^i}_K, \gauge{v^i}_L)}
		= \frac{v^i}{\gauge{v^i}_K}
		= \frac{v^i}{\gauge{v^i}_L}
\]
lie in $K \cap L$,
so $y = \sum_{i \in [k]} \lambda_i y^i$ also lies in $K \cap L$.
This shows
\[
	\gauge{x}_{\lowmean{q}{K}{L}}
		\leq m_{-q}(\gauge{x}_K, \gauge{x}_L)
		\leq \max \{ \gauge{x}_K, \gauge{x}_L \}
		\leq \gauge{x}_{\lowmean{q}{K}{L}}.
\]
Thus, $\gauge{x}_K = \gauge{x}_L = \gauge{x}_{\lowmean{q}{K}{L}}$ by $-q < -p \leq \infty$.

The equivalence in (ii) and the set inclusions from right to left are clear,
so it suffices to show that if $x \in \bd(\upmean{p}{K}{L})$,
then any $a \in N_{\upmean{q}{K}{L}}(x)$ satisfies $a^T x = h_K(a) = h_L(a)$.
Since, by Proposition~\ref{prop:support_gauge_fct}~(i) and Lemma~\ref{lem:p-mean_basics}~(i),

\[
	\gauge{a}_{\lowmean{-q}{K^\circ}{L^\circ}}
		= h_{\upmean{q}{K}{L}}(a)
		= h_{\upmean{p}{K}{L}}(a)
		= \gauge{a}_{\lowmean{-p}{K^\circ}{L^\circ}},
\]
(i) implies $\gauge{a}_{K^\circ} = \gauge{a}_{L^\circ} = \gauge{a}_{\lowmean{-p}{K^\circ}{L^\circ}}$.
We conclude $h_K(a) = h_L(a) = h_{\upmean{p}{K}{L}}(a) = a^T x$.

For (iii), we first observe that just $N_{\upmean{q}{K}{L}} \left( \frac{x}{\gauge{x}_{\lowmean{p}{K}{L}}} \right) \supsetneq \{0\}$ already suffices for $\gauge{x}_{\lowmean{p}{K}{L}} = \gauge{x}_{\upmean{q}{K}{L}}$.
For the other implication direction, we note that Lemma~\ref{lem:p-mean_basics} implies
\[
	\gauge{x}_{\upmean{q}{K}{L}}
		\leq \left\{ \begin{array}{c} \gauge{x}_{\lowmean{q}{K}{L}} \\ \gauge{x}_{\upmean{p}{K}{L}} \end{array} \right\}
		\leq \gauge{x}_{\lowmean{p}{K}{L}}.
\]
By assumption, equality holds from left to right.
We may therefore apply (i) to obtain
\[
	\gauge{x}_K
		= \gauge{x}_L
		= \gauge{x}_{\lowmean{q}{K}{L}}
		= \gauge{x}_{\upmean{p}{K}{L}}
		= \gauge{x}_{\upmean{q}{K}{L}}.
\]
Moreover, (ii) yields
\begin{align*}
	\{ 0 \}
		& \subsetneq N_{\upmean{q}{K}{L}} \left( \normal{x}{\upmean{p}{K}{L}} \right)
		= \left\{ a \in \R^n : a^T \normal{x}{\upmean{p}{K}{L}} = h_K(a) = h_L(a) \right\} \\
		& = N_K \left( \normal{x}{K} \right) \cap N_L \left( \normal{x}{L} \right)
		\subset N_{\upmean{\infty}{K}{L}} \left( \normal{x}{\upmean{p}{K}{L}} \right)
		\subset N_{\upmean{q}{K}{L}} \left( \normal{x}{\upmean{p}{K}{L}} \right).
\end{align*}
Thus, the claimed equality of sets follows from $\gauge{x}_{\upmean{p}{K}{L}} = \gauge{x}_{\lowmean{p}{K}{L}}$.
\end{proof}

For general $p < q$, it is not necessarily true that $\upmean{p}{K}{L} \subset \lowmean{q}{K}{L}$,
which is why Lemma~\ref{lem:common_boundary_different_p} does not cover the comparison of gauge functions in this case.
Accordingly, this combination of means requires different methods to handle when needed (see, e.g., Theorem~\ref{thm:equality_different_p}).

The following corollary collects consequences of the above lemmas,
which are sometimes more convenient to use than the rather technical lemmas themselves.

\begin{corollary}
\label{cor:common_boundary}
Let $K,L \in \CK^n_0$, $p, q \in [-\infty,\infty]$ with $p < q$, and $x \in \R^n \setminus \{ 0 \}$.
Then
\begin{enumerate}[(i)]
\item $x \in \lowmean{p}{K}{L} \cap \bd(\lowmean{q}{K}{L})$ if and only if $x \in \lowmean{-\infty}{K}{L} \cap \bd(\lowmean{q}{K}{L})$.
\item $x \in \upmean{p}{K}{L} \cap \bd(\upmean{q}{K}{L})$ if and only if $x \in \upmean{p}{K}{L} \cap \bd(\upmean{\infty}{K}{L})$.
\item $x \in \lowmean{p}{K}{L} \cap \bd(\upmean{q}{K}{L})$ if and only if $x \in \lowmean{-\infty}{K}{L} \cap \bd(\upmean{\infty}{K}{L})$.
\end{enumerate}
For $p \neq \pm\infty$, we have in addition:
\begin{enumerate}[(i)]
\item[(iv)] $\normal{x}{\lowmean{p}{K}{L}} \in \bd(\upmean{p}{K}{L})$
if and only if $N_K \left( \normal{x}{K} \right) \cap N_L \left( \normal{x}{L} \right) \allowbreak \supsetneq \{ 0 \}$.\newline
Independently of $K$ and $L$, this is always the case for at least one $x \in \R^n \setminus \{ 0 \}$.
\end{enumerate}
\end{corollary} \begin{proof}
The implications from right to left in (i)~--~(iii) are clear from Lemma~\ref{lem:p-mean_basics}~(iii);
the converses follow from (i)~--~(iii) in Lemma~\ref{lem:common_boundary_different_p}.
The equivalence in (iv) is already stated in Lemma~\ref{lem:common_boundary_same_p}.
For the claim that both statements apply for some $x$,
let $R > 0$ such that $K \subset R L$ and $x$ is a common boundary point of $K$ and $R L$.
Then there exists some $a \in \R^n \setminus \{ 0 \}$ such that $H_{(a,h_{R L}(a))}^\leq$ supports $K$ and $R L$ at $x$.
Thus, $a \in N_K \left( \normal{x}{K} \right) \cap N_L \left( \normal{x}{L} \right)$,
finishing the proof.
\end{proof}

With Lemma~\ref{lem:common_boundary_different_p} at hand,
we can now extend the initially mentioned results in \cite{BrDiMe} on the optimality of parts of \eqref{eq:forward_double_chain} to the full double-chain.
In \cite[Theorem~1.2]{BrDiMe}, it is shown that the minimum of bodies $K,L \in \CK^n_0$ is optimally contained in their maximum if and only if their harmonic mean is optimally contained in their arithmetic mean.
As a direct consequence, this remains true if $\lowmean{-1}{K}{L}$ and $\upmean{1}{K}{L}$ are respectively replaced with $\lowmean{p}{K}{L}$ and $\upmean{q}{K}{L}$ for some $p \in [-\infty,-1]$ and $q \in [1,\infty]$.
The theorem below extends this to general $p < q$, even when they are, say, on the same side of $0$.
In contrast, we cannot draw the same conclusion from $\lowmean{p}{K}{L} \subset^{opt} \upmean{p}{K}{L}$.

This extension of the equivalence could lead to progress in the study of the so-called \emph{asymmetry threshold of means} introduced in \cite{BrDiMe}.
It is defined as the largest Minkowski asymmetry of a Minkowski centered body for which its harmonic mean symmetrization is optimally contained in its arithmetic mean symmetrization,
that is,
\begin{equation}
\label{eq:threshold}
    \sup \{ s(K) : \lowmean{-1}{K}{-K} \subset^{opt} \upmean{1}{K}{-K}, K \in \CK^n_0 \text{ Minkowski centered} \}.
\end{equation}
As outlined above, this quantity does not change if $-1$ and $1$ are respectively replaced by some $p,q \in [-\infty,\infty]$ with $p < q$.

We additionally show in the theorem below that optimality at any point within the chain of either type of $p$-mean can always be extended toward an appropriate end of the chain,
though an extension toward both ends cannot be expected in general as outlined in the discussion preceding \cite[Theorem~1.5]{BrDiMe}.
In summary, we show that the equivalence in \cite[Theorem~1.2]{BrDiMe} can be extended to any pair of means from \eqref{eq:forward_double_chain} where one is strictly to the top right of the other, but not necessarily to any other pairs.

Our proof of the results about the optimality in \eqref{eq:forward_double_chain} 
requires the following characterization of optimal containment taken from \cite[Theorem~$2.3$]{BrKoe}.

\begin{proposition}
\label{prop:opt_cont}
Let $K \in \CK^n$ and $L \in \CK^n_0$.
Then $K \subset^{opt} L$ holds if and only if the following conditions are satisfied:
\begin{enumerate}[(i)]
\item $K \subset L$.
\item There exist $m \in \{2, \ldots, n+1\}$ and $a^1, \ldots, a^m \in \R^n \setminus \{0\}$ such that $h_K(a^i) = h_L(a^i)$ for every $i \in [m]$ and $0 \in \conv(\{ a^1, \ldots, a^m \})$.
\end{enumerate}
\end{proposition}

We now present our theorem on the optimal containment between general $p$-means.
Remember for infinite values of $p$ in part (iv) that Theorem~\ref{thm:same_p} asserts an even stronger statement, namely that the upper and lower $p$-mean always coincide.

\begin{theorem}
Let $K, L \in \CK^n_0$ and $p,q \in [-\infty,\infty]$ with $p < q$.
Then
\begin{enumerate}[(i)]
\item $\lowmean{p}{K}{L} \subset^{opt} \lowmean{q}{K}{L}$ if and only if $\lowmean{-\infty}{K}{L} \subset^{opt} \lowmean{q}{K}{L}$.
\item $\upmean{p}{K}{L} \subset^{opt} \upmean{q}{K}{L}$ if and only if $\upmean{p}{K}{L} \subset^{opt} \upmean{\infty}{K}{L}$.
\item $\lowmean{p}{K}{L} \subset^{opt} \upmean{q}{K}{L}$ if and only if $\lowmean{-\infty}{K}{L} \subset^{opt} \upmean{\infty}{K}{L}$.
\item $\lowmean{p}{K}{L} \subset^{opt} \upmean{p}{K}{L}$ if both means are $0$-symmetric.
This applies in particular if $K$ and $L$ are $0$-symmetric, or if $L = -K$.
\end{enumerate}
\end{theorem} \begin{proof}
The implications from right to left in the first three cases are direct consequences of Lem\-ma~\ref{lem:p-mean_basics}~(iii).
We turn to their converses:
For (i), Lemma~\ref{lem:common_boundary_different_p}~(i) shows that if $x \in \R^n$
is a common boundary point of $\lowmean{p}{K}{L}$ and $\lowmean{q}{K}{L}$,
then $x$ also lies in $\lowmean{-\infty}{K}{L}$.
Hence, any halfspace that supports $\lowmean{p}{K}{L}$ and $\lowmean{q}{K}{L}$ at $x$ also supports $\lowmean{-\infty}{K}{L}$ at $x$.
Now, Proposition~\ref{prop:opt_cont} implies (i).

For (ii), Lemma~\ref{lem:common_boundary_different_p}~(ii) shows that if $x \in \R^n$
is a common boundary point of $\upmean{p}{K}{L}$ and $\upmean{q}{K}{L}$,
then any halfspace that supports $\upmean{p}{K}{L}$ and $\upmean{q}{K}{L}$ at $x$ also supports $\upmean{\infty}{K}{L}$ at $x$.
Using Proposition~\ref{prop:opt_cont} again implies (ii).

For (iii), Lemma~\ref{lem:common_boundary_different_p}~(iii) shows that
any common boundary point $x \in \R^n$ of $\lowmean{p}{K}{L}$ and $\upmean{q}{K}{L}$ also lies in $\lowmean{-\infty}{K}{L}$,
and that any hyperplane that supports $\upmean{q}{K}{L}$ at $x$ also supports $\upmean{\infty}{K}{L}$ at $x$.
A third time, Proposition~\ref{prop:opt_cont} yields the desired result.

Finally, if $\lowmean{p}{K}{L}$ and $\upmean{p}{K}{L}$ are both 0-symmetric,
then a single common boundary point is sufficient to prove the optimality of the containment in (iv) by Proposition~\ref{prop:opt_cont}.
Such a point is guaranteed to exist by Corollary~\ref{cor:common_boundary}~(iv) if $p \neq \pm \infty$.
The case of $p = \pm \infty$ is omitted here since it is covered by Theorem~\ref{thm:same_p},
which we prove in the next section.
\end{proof}

\section{Equality Cases Between \texorpdfstring{$p$}{p}-Means}
\label{sec:equality_cases}

In this section, we answer the questions about equality cases between $p$-means that naturally arise from the basic containment relations shown in Lemma~\ref{lem:p-mean_basics}.
First, we verify Theorem~\ref{thm:same_p}, providing the equality case for Lemma~\ref{lem:p-mean_basics}~(ii).
We then turn to Lemma~\ref{lem:p-mean_basics}~(iii), where we consider increasing $p$ or increasing the sets $K$ and $L$ separately.
While we obtain the same strict monotonicity in $p$ as for real numbers,
we shall see below that the strict monotonicity in the set arguments is valid only for certain values of $p$.
Additionally, we characterize the equality cases when different types of $p$-means and different $p$-values are mixed at the same time.
While one combination is immediately covered by the above results, the other requires a more involved approach provided in Section~\ref{sec:radii}.

We begin with establishing Theorem~\ref{thm:same_p}.
Most of the technical details required are already covered by Lemma~\ref{lem:common_boundary_same_p}.
The final ingredient is an additional result
that essentially characterizes when two bodies $K,L \in \CK^n_0$
are dilates of each other.
Despite appearing quite natural,
it is surprisingly hard to find related literature.
The only result we could locate in this direction is \cite[Theorem~1]{Ha},
where the special case for $L = -K$ is essentially established as part of the proof.
While the approach there could be extended to general $L$
and can also be found in traces of the proof presented below,
we demonstrate slightly more general techniques
that allow easier extensions of the result (see Remark~\ref{rem:dilation_char} below).
Let us also point out that the result below is part of
the master's thesis \cite{Grun}.

\begin{theorem}
\label{thm:dilates}
Let $K,L \in \CK^n_0$.
Then the following are equivalent:
\begin{enumerate}[(i)]
\item For all $x \in \R^n \setminus \{0\}$, there exist parallel hyperplanes supporting $K$ at $\normal{x}{K}$ and $L$ at $\normal{x}{L}$.
\item $n=1$, or $K$ and $L$ are dilates of each other.
\end{enumerate}

\end{theorem}
\begin{proof}
Since (ii) clearly implies (i),
we only need to show that $K$ and $L$ are dilates of each other if (i) is satisfied for $n \geq 2$.
To this end, it suffices to show for all $x,y \in \R^n \setminus \{0\}$ that
\[
    \frac{\gauge{x}_L}{\gauge{x}_K}
    = \frac{\gauge{y}_L}{\gauge{y}_K}.
\]
Since $n \geq 2$, there exists a Lipschitz continuous curve $\gamma \colon [0,1] \to \R^n \setminus \{0\}$ with
$\gamma(0) = x$ and $\gamma(1) = y$.
We define the curve $\gamma_K \colon [0,1] \to \bd(K)$, $\gamma_K(t) = \normal{\gamma(t)}{K}$, and analogously $\gamma_L$ for $L$.
The function $\gauge{\cdot}_K$ is convex and therefore Lipschitz on the compact set $\gamma([0,1])$.
Thus, $\gamma_K$ is the quotient of two Lipschitz maps,
where $\gauge{\gamma(t)}_K$ is
additionally bounded from above and below by some positive constants.
Consequently, $\gamma_K$ is also Lipschitz.
In the same way, it follows that $\gamma_L$ and
\[
    \lambda \colon [0,1] \to \R,
    \lambda(t) = \frac{\gauge{\gamma(t)}_L}{\gauge{\gamma(t)}_K}
\]
are Lipschitz.

Next, suppose that $\gamma_K$ is differentiable at some $t \in (0,1)$.
Let $a \in \R^n$ such that $a^T \gamma_K(t) = h_K(a)$.
Then
\[
    a^T \gamma_K'(t)
    = \lim_{0 \neq h \to 0} \frac{a^T \gamma_K(t+h) - a^T \gamma_K(t)}{h},
\]
where $a^T \gamma_K(t+h) \leq h_K(a) = a^T \gamma_K(t)$
shows that the difference quotient in the limit is non-positive for $h > 0$ and non-negative for $h < 0$.
It follows $a^T \gamma_K'(t) = 0$.

Now, choose any $t \in (0,1)$.
By the assumption (i) and $0 \in \inte(K) \cap \inte(L)$, there exists some $a \in \R^n$ with $a^T \gamma_K(t) = h_K(a) > 0$ and $a^T \gamma_L(t) = h_L(a) > 0$.
There further exists some open set $U \subset [0,1]$ with $t \in U$ such that $a^T \gamma(s) > 0$ for all $s \in U$ and thus
\[
    \lambda(s)
    = \frac{a^T \gamma(s) \gauge{\gamma(s)}_L}{a^T \gamma(s) \gauge{\gamma(s)}_K}
    = \frac{a^T \gamma_K(s)}{a^T \gamma_L(s)}.
\]
Consequently, if $\gamma_K$ and $\gamma_L$ are both differentiable at $t$,
then so is $\lambda$.
In this case, $a^T \gamma_K'(t) = 0 = a^T \gamma_L'(t)$ by the above shows $\lambda'(t) = 0$.

Finally, the Fundamental Theorem of Calculus for Lebesgue Integration shows that $\gamma_K$ and $\gamma_L$ are, as Lipschitz maps, differentiable almost everywhere.
Hence, $\lambda$ has derivative $0$ almost everywhere.
Since $\lambda$ is also Lipschitz, applying the Fundamental Theorem of Calculus for Lebesgue Integration again shows that $\lambda$ is constant.
In particular,
\[
    \frac{\gauge{x}_L}{\gauge{x}_K}
    = \lambda(0)
    = \lambda(1)
    = \frac{\gauge{y}_L}{\gauge{y}_K}
\]
is true as required.
\end{proof}

\begin{remark}
\label{rem:dilation_char}
An advantage of the above proof over the idea presented in \cite{Ha} is that it can be generalized more easily.
For example, we could assume the property in (i) only for those points in some positively homogeneous set $C \subset \R^n \setminus \{0\}$ (i.e., a set with $(0,\infty) C = C$) such that any two points in $C$ are connected by a Lipschitz continuous (or even just absolutely continuous) curve that stays within $C$.
This applies in particular if $C \cup \{0\}$ is a convex cone that is not a line.
With an essentially identical proof, we would then conclude that $K \cap C$ and $L \cap C$ must be dilates of each other.

It would be interesting to know if the above assumptions can be relaxed even further to $C \subset \R^n \setminus \{0\}$ being positively homogeneous and connected.
This would clearly be a minimal set of assumptions on $C$ to still be able to conclude that $K \cap C$ and $L \cap C$ are dilates of each other.
Let us point out for $n = 2$ that $C$ being positive homogeneous and connected is already enough to obtain the stronger connectivity assumption from above,
but this implication may fail in higher dimensions.
\end{remark}

We are now ready to analyze the equality cases of $p$-means,
beginning with upper and lower $p$-means for the same $p$-value.

\begin{proof}[Proof of Theorem~\ref{thm:same_p}]
We first consider the cases $p = \pm \infty$.
By Lemma~\ref{lem:p-mean_basics}~(i)~and~(ii),
we only need to verify the inclusion $\upmean{-\infty}{K}{L} \subset \lowmean{-\infty}{K}{L} = K \cap L$.
To this end, let $x \in \upmean{-\infty}{K}{L}$ and $a \in \mS_2$.
Then $a^T x \leq m_{-\infty}(h_K(a),h_L(a)) = \min \{ h_K(a), h_L(a) \}$.
Since $a$ has been chosen arbitrarily, we obtain $x \in K \cap L$.

From now on, we assume $p \neq \pm \infty$.
By Corollary~\ref{cor:common_boundary}~(iv),
(i) is equivalent to the fact that for all $x \in \mS_2$, $K$ and $L$ can be supported
by parallel hyperplanes at $\normal{x}{K}$ and $\normal{x}{L}$, respectively.
This, in turn, is equivalent to $n = 1$ or $K$ and $L$ being dilates of each other according to Theorem~\ref{thm:dilates}.
\end{proof}

We turn to the equality cases for different values of $p$,
extending previously known results \cite{Fipolarmeans,Fipmeans}.
The following theorem is a direct corollary of Theorem~\ref{thm:stability}, which is proven in the next section.
We state it here explicitly for the sake of completeness.
Let us point out that unless $(M_p,M_q) = (\overline{M}_p, \underline{M}_q)$,
this result could also be derived from Corollary~\ref{cor:common_boundary}.
However, for this particular pair of $p$-means where we have no natural containment relation in general,
we need to resort to our methods based on the covering radius presented in the next section.

\begin{theorem}
\label{thm:equality_different_p}
Let $K,L \in \CK^n_0$, $p,q \in [-\infty,\infty]$ with $p < q$, and $M_p \in \{ \overline{M}_p, \underline{M}_ p\}$,
$M_q \in \{ \overline{M}_q, \underline{M}_q \}$.
Then the following are equivalent:
\begin{enumerate}[(i)]
\item $M_p(K,L) = M_q(K,L)$.
\item $K = L$.
\end{enumerate}
\end{theorem}

Lastly, we consider the monotonicity of $p$-means in their set-arguments.
For $p$-means of reals numbers, it is well-known that they are strictly increasing in both arguments when $p \in \R$.
While this immediately leads to strict monotonicity of $p$-means of convex bodies also in higher dimensions for those values of $p$ where Firey's original definitions apply,
the strict monotonicty does not extend to all values of $p$ as observed in Example~\ref{ex:polytope_redundant}.

\begin{theorem}
Let $K, K', L, L' \in \CK^n_0$ with $K \subset K'$, $L \subset L'$, and either $n = 1$ and $p \in \R$, or $n \geq 2$ and $p \in [1,\infty)$. Then the following are equivalent:
\begin{enumerate}[(i)]
\item $\lowmean{-p}{K}{L} = \lowmean{-p}{K'}{L'}$.
\item $\upmean{p}{K}{L} = \upmean{p}{K'}{L'}$.
\item $K = K'$ and $L = L'$.
\end{enumerate}
For all other combinations of $n \in \N$ and $p \in [-\infty,\infty]$,
only the implications from (iii) to (i) and (ii) hold in general.
\end{theorem} \begin{proof}
(iii) is clearly sufficient for (i) and (ii).
The implications from (i) and (ii) to (iii) when $n=1$ and $p\in \R$ are well-known (see Remark~\ref{rem:mp}).
For $p \in [1,\infty)$,
we use that $h_{\upmean{p}{K}{L}} = m_p(h_K,h_L)$ by
the discussion following Definition~\ref{def:p-means}.
An analogous identity applies for $K'$ and $L'$.
Hence, $m_p(h_K,h_L) = m_p(h_{K'},h_{L'})$, which by $p \in [1,\infty)$ implies $h_K = h_{K'}$ and $h_L = h_{L'}$.
This gives the implication from (ii) to (iii).
For the implication from (i) to (iii), we use an analogous argument involving gauge functions.

Next, we provide examples where (i) or (ii) are satisfied but the other two statements in the lemma are violated when $p = \pm \infty$, or $n \geq 2$ and $p \in (-\infty,1)$.
Theorem~\ref{thm:same_p} and Lemma~\ref{lem:p-mean_basics}~(i) show for any $L = L' = K \subsetneq K'$ that $L = K \cap L = \lowmean{-\infty}{K}{L} = (\lowmean{\infty}{K^\circ}{L^\circ})^\circ$,
where these means do not change if we replace $K$ by $K'$ and $L$ by $L'$.
However, all of $K = \conv(K \cup L) = \upmean{\infty}{K}{L} = (\upmean{-\infty}{K^\circ}{L^\circ})^\circ$ clearly do change under this replacement.
This gives an example where (i) is satisfied but (ii) and (iii) are violated for $p = \pm \infty$.

For $n \geq 2$ and $p \in (-\infty,1)$, Example~\ref{ex:polytope_redundant} already provides an instance of $K,K',L,L'$ where $\lowmean{-p}{K}{L} = \lowmean{-p}{K'}{L'}$ but $\upmean{p}{K}{L} \subsetneq \upmean{p}{K'}{L'}$.
Therefore, (i) is again satisfied, but (ii) and (iii) are violated.

Lastly, we note that whenever we have an example where (i) is satisfied but (ii) and (iii) are violated for some $p$, then going to the polar bodies with Lemma~\ref{lem:p-mean_basics}~(i) gives an example where (ii) is satisfied but (i) and (iii) are violated for the same value of $p$.
This covers all remaining cases.
\end{proof}

\section{Covering Radii}
\label{sec:radii}

In this section, we verify the bounds on the covering radii between $p$-means claimed in Theorems~\ref{thm:lower_bound}~and~\ref{thm:upper_bounds}.
Based on the former, we then also show the stability estimates in Theorem~\ref{thm:stability}.
Note that several of the inequalities proven here generalize results obtained in \cite{BrDiMe},
which we discuss in detail in Remark~\ref{rem:BDM_comparison} toward the end of this section.

\begin{remark}
\label{rem:R0_basics}
Before we begin with the proofs, let us discuss the use of the covering radius and covering distance for the results throughout this section.
We want to compare $p$-means based on a metric that is invariant under invertible linear transformations and does not introduce translations,
so that we can properly concentrate on the shapes, sizes, and positioning of the bodies relative to each other rather than depending on these properties of the individual bodies.
These conditions are evidently satisfied by the covering distance.
Moreover, it is straightforward to see for closed convex sets $K,L \subset \R^n$ containing $0$ that $R_0(K^\circ,L^\circ) = R_0(L,K)$,
so in particular $R_0^{\max}(K^\circ,L^\circ) = R_0^{\max}(K,L)$.
An easy volume argument further shows for $K,L \in \CK^n_0$ that $R_0^{\max}(K,L) \geq 1$, with equality if and only if $K = L$.
Thus, $R_0^{\max}$ is quite obviously a (multiplicative) metric on $\CK^n_0$ that satisfies several desired properties and allows us to capitalize on the dual nature of upper and lower $p$-means.

Let us also point out that the covering distance unifies some desirable properties of the well-known Hausdorff and Banach--Mazur distances.
On the one hand, it induces the same topology as the Hausdorff distance and can properly distinguish whether convex bodies coincide (without depending on their common scaling).
On the other hand, it enjoys some invariance properties of the Banach--Mazur distance (without collapsing affine or linear classes into single objects) and is closely related to this distance in terms of their definitions via containment problems,
such that they can be easily related to each other in appropriate cases (see Section~\ref{sec:distance}).

\pagebreak
Lastly, recall that we want to introduce dilation factors to sharpen or reverse the natural containment relations between $p$-means.
To this aim, the covering distance follows the spirit of \cite{BrDiMe}.
The connection to \cite{BrDiMe} becomes apparent from the following comparison of the covering radius and the circumradius,
which is also a preparation for the results in Section~\ref{sec:distance}.
For $K,L \in \CK^n_0$, the covering radius $R_0(K,L)$ coincides with the circumradius $R(K,L)$ if and only if $0$ is a circumcenter.
Unlike the former, the circumradius is invariant under translations of $K$ and $L$.
Furthermore, \textup{\cite[Lemma~2.1]{XiLe}} shows that there always exist translates $K'$ and $L'$ of fulldimensional bodies $K$ and $L$
such that $R_0(K',L') = R(K,L)$, $R_0(L',K') = R(L,K)$, and either $0 \in \inte(K' \cap L')$ or $0 \in \bd(K') \cap \bd(L')$.
It is not hard to see that this yields
\begin{equation}
\label{eq:maxR}
	\max \{ R(K,L), R(L,K) \}
    = \min \{ R_0^{\max}(K - x, L - y) : x \in K, y \in L \},
\end{equation}
where the minimum might not be attained for any pair of $x \in \inte(K)$ and $y \in \inte(L)$.
If $K$ and $L$ are $0$-symmetric, or $K$ is Minkowski centered and $L = -K$, then all covering radii considered above are actually circumradii. Consequently, equality holds in \eqref{eq:maxR} even without translations.
In particular, in the latter case we have $R_0^{\max}(K,-K) = s(K)$.
\end{remark}

We begin the proofs with some technical results.
The assertion of the first one
is clear from Proposition~\ref{prop:support_gauge_fct}~(v)
if $f = h_{K_f} \vert_{\mS_2}$ and $g = h_{K_g} \vert_{\mS_2}$ are positive and real valued functions.
However, in the context of upper $p$-means or Remark~\ref{rem:R0_boundary} below,
we sometimes only have an upper bound on the support functions or need to allow the values $0$ and $\infty$.
With the following lemma, we are able to replace the support functions by such upper bounds,
assuming the corresponding convex sets as defined directly below do not differ from the original sets (as is the case for upper $p$-means).

\begin{lemma}
\label{lem:R0_support_fct}
Let $f,g \colon \mS_2 \to [0,\infty]$ be any maps.
Define $K_f \coloneqq \bigcap_{a \in \mS_2} H_{(a,f(a))}^\leq$ and $K_g$ analogously for $g$.
Then
\[
	R_0(K_f,K_g)
    \leq \sup \left\{ \frac{f(a)}{g(a)} : a \in \mS_2 \right\},
\]
with equality if $f = h_{K_f} \vert_{\mS_2}$,
where we write
\begin{equation}
\label{eq:frac_def}
	\frac{\mpvara}{\mpvarb}
		\coloneqq \inf \{ \lambda > 0 : \mpvara \leq \lambda \mpvarb \}
		= \begin{cases}
			0,				& \text{if $\mpvara = 0$ or $\mpvarb = \infty$}, \\
			\infty,				& \text{if $\mpvara \neq 0$ and $\mpvarb = 0$, or $\mpvara = \infty$ and $\mpvarb \neq \infty$}, \\
			\frac{\mpvara}{\mpvarb},	& \text{else},
		\end{cases}
\end{equation}
for any $\mpvara, \mpvarb \in [0,\infty]$.
\end{lemma}

\begin{proof}
Let $\sigma$ be the supremum in the statement of the lemma.
If $\sigma = \infty$, the inequality is clear.
Otherwise, we have for any $a \in \mS_2$ that $\frac{f(a)}{g(a)} \leq \sigma < \infty$
and thus with \eqref{eq:frac_def} for any $\varepsilon > 0$ also $f(a) \leq (\sigma + \varepsilon) g(a)$.
Hence,
\[
	K_f \subset H_{(a,f(a))}^\leq
		\subset H_{(a,(\sigma + \varepsilon) g(a))}^\leq
		= (\sigma + \varepsilon) H_{(a,g(a))}^\leq,
\]
which proves $K_f \subset (\sigma + \varepsilon) K_g$ and therefore $R_0(K_f,K_g) \leq \sigma + \varepsilon$ by intersecting over all $a \in \mS_2$.
Letting $\varepsilon$ tend to $0$ yields the claimed inequality.

Now, assume $f = h_{K_f} \vert_{\mS_2}$.
Let $(a^i)_{i \in \mathbb{N}}$ be a sequence with $a^i \in \mS_2$ for all $i$ and $\frac{f(a^i)}{g(a^i)} \to \sigma$ $(i \to \infty)$.
By $0 \leq f = h_{K_f} \vert_{\mS_2}$, there exists for each $i \in \N$
a sequence $(x^{i,m})_{m \in \N}$ with $x^{i,m} \in K_f$ for all $m$ and $0 \leq (a^i)^T x^{i,m} \to f(a^i)$ ($m \to \infty$).
With \eqref{eq:frac_def}, we obtain
\[
	\gauge{x^{i,m}}_{K_g}
		\geq \gauge{x^{i,m}}_{H_{(a^i,g(a^i))}^{\leq}}
		= \frac{(a^i)^T x^{i,m}}{g(a^i)},
\]
where the last term converges to $\frac{f(a^i)}{g(a^i)}$ for $m \to \infty$.
Since $x^{i,m} \in K_f$ for all $i,m$,
\[
	R_0(K_f,K_g)
	\geq \sup \left\{ \gauge{x^{i,m}}_{K_g} : i,m \in \N \right\}
	\geq \sup \left\{ \frac{f(a^i)}{g(a^i)} : i \in \N \right\}
	= \sigma
\]
shows the required reverse inequality.
\end{proof}

Our second set of technical results concerns for $p,q \in [-\infty,\infty]$ the function
\[
    \Phi_{p,q} : (0,\infty) \to (0,\infty),
    \Phi_{p,q}(x) = \frac{m_q(x,1)}{m_p(x,1)}.
\]
Since it comes up in Theorems~\ref{thm:stability}~--~\ref{thm:upper_bounds},
it is natural that we need some analytical properties of this function.
The first few observations are collected in the following remark.

\begin{remark} \label{rem:phi_pq}
Let $p,q \in [-\infty,\infty]$.
For $\mpvara,\mpvarb > 0$, we obtain as immediate consequences of the properties in Remark~\ref{rem:mp} that
\begin{equation}
\label{eq:phi_inv_arg}
	\Phi_{p,q} \left( \frac{\mpvara}{\mpvarb} \right)
		= \frac{m_q(\mpvara/\mpvarb,1)}{m_p(\mpvara/\mpvarb,1)}
		= \frac{m_q(\mpvara,\mpvarb)}{m_p(\mpvara,\mpvarb)}
		= \frac{m_q(1,\mpvarb/\mpvara)}{m_p(1,\mpvarb/\mpvara)}
		= \Phi_{p,q} \left( \frac{\mpvarb}{\mpvara} \right)
\end{equation}
as well as
\[
	\Phi_{-q,-p}(\mpvara)
		= \frac{m_{-p}(\mpvara,1)}{m_{-q}(\mpvara,1)}
		= \frac{m_{-p}(1/\mpvara,1)}{m_{-q}(1/\mpvara,1)}
		= \frac{m_q(\mpvara,1)}{m_p(\mpvara,1)}
		= \Phi_{p,q}(\mpvara).
\]
Moreover, $\Phi_{p,q}$ is clearly differentiable on $(1,\infty)$ as the quotient of two differentiable functions.
If $p$ and $q$ do not lie in $\{-\infty, 0, \infty \}$, its derivative is
\begin{align*}
	\Phi_{p,q}'(\mpvara)
	& = \frac{
        \left( \frac{\mpvara^q + 1}{2} \right)^{\frac{1}{q}-1} \frac{\mpvara^{q-1}}{2} \, m_p(\mpvara,1)
		- \left( \frac{\mpvara^p + 1}{2} \right)^{\frac{1}{p}-1} \frac{\mpvara^{p-1}}{2} \, m_q(\mpvara,1)
    }{m_p(\mpvara,1)^2} \\
	& = \frac{
        m_q(\mpvara,1) \, m_p(\mpvara,1) \left( \frac{\mpvara^{q-1}}{\mpvara^q + 1} - \frac{\mpvara^{p-1}}{\mpvara^p + 1} \right)
    }{m_p(\mpvara,1)^2}
	= \frac{\Phi_{p,q}(\mpvara)}{\mpvara} \left( \frac{1}{1+\mpvara^{-q}} - \frac{1}{1+\mpvara^{-p}} \right).
\end{align*}
Defining $\frac{1}{1+\mpvara^{-\infty}} \coloneqq 1$ and $\frac{1}{1+\mpvara^{\infty}} \coloneqq 0$ for all $\mpvara > 1$,
it can easily be verified that the final formula above for $\Phi_{p,q}'(\mpvara)$ is correct for all $p,q \in [-\infty,\infty]$.
Hence, $\Phi_{p,q}'(\mpvara) \geq 0$ for $\mpvara > 1$ if and only if $p \leq q$.
Together with \eqref{eq:phi_inv_arg}, $\Phi_{p,q}$ is increasing on $[1,\infty)$ and decreasing on $(0,1]$ for $p \leq q$.
\end{remark}

Our last technical result provides some approximations to $\Phi_{p,q}$.
We utilize them for our proof of Theorem~\ref{thm:stability} at the end of this section.
They are needed mainly because we use the inverse function of $\Phi_{p,q}$ restricted to $[1,\infty)$ when $p<q$ (simply denoted $\Phi_{p,q}^{-1} : \Phi_{p,q}([1,\infty)) \to [1,\infty)$ instead of $(\Phi_{p,q} |_{[1,\infty)})^{-1}$ in the following),
where there is no closed form for this inverse function in general.

\begin{lemma} \label{lem:stability}
Let $p,q \in \R$ with $p < q$ and $\mpvara \in [1,\infty)$.
Then
\[
    \Phi_{p,q}(\mpvara)
    \leq 1 + \frac{q-p}{8} (\mpvara-1)^2.
\]
Moreover, if $\delta \in [0,1)$ and $\mpvara \leq 1 + \frac{\delta}{\max\{ |p|, |q| \} + 3/2}$, then
\[
    \Phi_{p,q}(\mpvara)
    \geq 1 + \frac{q-p}{8} (1-\delta) (\mpvara-1)^2.
\]
\end{lemma}

Let us point out that the $3/2$ in the denominator in the assumption for the second part is the smallest possible constant usable in our setup when $p$ and $q$ both go to $0$.

\begin{proof}
We begin with the upper bound,
for which we claim in case of $\intvar>1$ that
\begin{equation} \label{eq:stability_lemma_claim1}
    \Phi_{p,q}'(\intvar)
    \leq \frac{q-p}{4} (\intvar-1).
\end{equation}
Integrating both sides
over $\intvar \in [1,\mpvara]$ then
verifies the upper bound in the lemma since $\Phi_{p,q}(1) = 1$.
By Remark~\ref{rem:phi_pq},
we have
\[
    \Phi_{p,q}'(\intvar)
    = \frac{\Phi_{p,q}(\intvar)}{\intvar} \left( \frac{1}{1+\intvar^{-q}} - \frac{1}{1+\intvar^{-p}} \right),
\]
and clearly
\[
    \Phi_{p,q}(\intvar)
    \leq \frac{m_\infty(\intvar,1)}{m_{-\infty}(\intvar,1)}
    = \intvar.
\]
Therefore, the claim \eqref{eq:stability_lemma_claim1} follows once we verify that
\begin{equation} \label{eq:stability_lemma_mon1}
    \frac{1}{1+\intvar^{-q}} - \frac{1}{1+\intvar^{-p}}
    \leq \frac{q-p}{4} (\intvar-1).
\end{equation}

To this end, we define the function
\[
    h_\intvar \colon \R \to \R,
    h_\intvar(r) = \frac{1}{1+\intvar^{-r}}.
\]
By the mean value theorem, there exists some $r \in (p,q)$ such that
\[
    \frac{1}{1+\intvar^{-q}} - \frac{1}{1+\intvar^{-p}}
    = h_\intvar(q) - h_\intvar(p)
    = (q-p) h_\intvar'(r)
    = (q-p) \frac{\ln(\intvar) \intvar^{-r}}{(1+\intvar^{-r})^2}.
\]
Using $\ln(\intvar) \leq \intvar-1$ and $4 \intvar^{-r} \leq (1+\intvar^{-r})^2$ yields \eqref{eq:stability_lemma_mon1} and therefore \eqref{eq:stability_lemma_claim1}.

We turn to the lower bound,
for which we may suppose $\delta \in (0,1)$ and use an analogous but slightly more involved approach than the above.
This time, we claim for $\intvar \in I \coloneqq (1,1 + \frac{\delta}{\max\{ |p|, |q| \} + 3/2})$
that
\begin{equation} \label{eq:stability_lemma_claim2}
    \Phi_{p,q}'(\intvar)
    \geq \frac{q-p}{4} (1-\delta) (\intvar-1).
\end{equation}
Again integrating both sides
over $\intvar \in [1,\mpvara]$ then yields the lower bound in the lemma since $\Phi_{p,q}(1) = 1$.
By Remark~\ref{rem:phi_pq} and $\Phi_{p,q}(\intvar) \geq 1$,
it suffices to show for $\intvar \in I$ that
\begin{equation} \label{eq:stability_lemma_mon2}
    \frac{1}{\intvar} \left( \frac{1}{1+\intvar^{-q}} - \frac{1}{1+\intvar^{-p}} \right)
    \geq \frac{q-p}{4} (1-\delta) (\intvar-1).
\end{equation}

As before, there exists some $r \in (p,q)$ satisfying
\[
    \frac{1}{1+\intvar^{-q}} - \frac{1}{1+\intvar^{-p}}
    = (q-p) \frac{\ln(\intvar) \intvar^{-r}}{(1+\intvar^{-r})^2},
\]
with which we can rewrite \eqref{eq:stability_lemma_mon2} as
\[
    \frac{\ln(\intvar)}{\intvar (\intvar-1)}
    \geq (1-\delta) \frac{(1+\intvar^{-r})^2}{4 \intvar^{-r}}.
\]
Now, note that
\[
    \frac{(1+\intvar^{-r})^2}{4 \intvar^{-r}}
    = m_{\frac{1}{2}}(\intvar^r,\intvar^{-r})
    \leq \intvar^{|r|}.
\]
Moreover, the assumption $\intvar \in I$ implies
\[
    1-\delta
    \leq 1 - (\intvar-1) \left( |r| + \frac{3}{2} \right),
\]
so we can establish \eqref{eq:stability_lemma_mon2} by verifying for all $\intvar \in (1,\infty)$ and $r \in [0,\infty)$ that
\begin{equation} \label{eq:stability_lemma_mon3}
    \frac{\ln(\intvar)}{\intvar (\intvar-1)}
    \geq \intvar^{r} \left( 1-(\intvar-1) \left( r+\frac{3}{2} \right) \right).
\end{equation}
The right-hand term is differentiable in $r > 0$ with derivative
\[
    \intvar^r \left( \ln(\intvar) + 1 - \intvar - (\intvar-1) \left( r + \frac{3}{2} \right) \ln(\intvar) \right)
    \leq 0
\]
by $\ln(\intvar) + 1 \leq \intvar$.
Hence, it suffices to show \eqref{eq:stability_lemma_mon3} for $r=0$,
in which case it can be rearranged to
\[
    \ln(\intvar) - \intvar (\intvar-1) \left( 1 - \frac{3}{2} (\intvar-1) \right) \geq 0.
\]
The left-hand term attains the value $0$ at $\intvar=1$ and is differentiable in $\intvar \in (1,\infty)$ with derivative
\[
    \frac{(\intvar-1)^2 (9\intvar+2)}{2\intvar}
    \geq 0.
\]
Altogether, we finally obtain \eqref{eq:stability_lemma_mon3}, thus
\eqref{eq:stability_lemma_mon2},
and consequently \eqref{eq:stability_lemma_claim2}.
\end{proof}

We are now ready to prove how well $p$-means approximate each other in terms of covering radii.
Let us begin with Theorem~\ref{thm:lower_bound} and parts (i)~\m&~(ii) of Theorem~\ref{thm:upper_bounds},
whose proofs go hand in hand.
The upper bounds on $R_0(\upmean{p}{K}{L},\lowmean{q}{K}{L})$ are dealt with separately,
as this combination of means behaves differently from the other three and requires more work (compare with Section~\ref{sec:boundary}, where it was also omitted).

Throughout the proofs in this section,
we write $R_0^{\max} \coloneqq R_0^{\max}(K,L)$ whenever it is clear that we mean the covering distance of $K$ and $L$.
Moreover, for the upcoming proof specifically,
we shorten $R_0(\lowmean{p}{K}{L},\upmean{q}{K}{L})$ to $R_0(\underline{p},\overline{q})$
and use analogous abbreviations for the other combinations of upper and lower $p$-means.
The bars are omitted for $\pm \infty$ since the respective means in that case are equal by Theorem~\ref{thm:same_p}.

\begin{proof}[Proof of Theorem~\ref{thm:lower_bound} and Theorem~\ref{thm:upper_bounds}~(i)~\m&~(ii)]
First, observe that Lemma~\ref{lem:p-mean_basics}~(ii) shows
\begin{equation}
\label{eq:R0_p-means_relation}
	R_0(\underline{p},\overline{q})
		\leq \left\{ \begin{array}{c} R_0(\overline{p},\overline{q}) \\ R_0(\underline{p},\underline{q}) \end{array} \right\}
		\leq R_0(\overline{p},\underline{q}).
\end{equation}
With Lemma~\ref{lem:p-mean_basics}~(iii), this already verifies Theorem~\ref{thm:upper_bounds}~(ii).
Moreover, $R_0^{\max} = R_0^{\max}(K,L) \geq 1$ yields
\[
	\left. \begin{array}{c} K \\ L \end{array} \right\}
		\subset \conv(K \cup L)
		\subset R_0^{\max} \cdot (K \cap L)
		\subset \left\{ \begin{array}{c} R_0^{\max} K \\ R_0^{\max} L \end{array} \right.
\]
and therefore
\begin{equation}
\label{eq:R0_infty_means}
	R_0^{\max}
		= R_0(\infty,-\infty)
		\leq R_0(\infty,\underline{p}) R_0(\underline{p},\overline{q}) R_0(\overline{q},-\infty).
\end{equation}

Next, with $R_0(B^\circ,A^\circ) = R_0(A,B)$ for closed convex $A,B \subset \R^n$ containing $0$,
Lemma~\ref{lem:R0_support_fct} yields
\begin{equation}
\label{eq:R0max_support_fct}
	R_0^{\max}
	= \sup \left\{ \frac{h_K(a)}{h_L(a)}, \frac{h_L(a)}{h_K(a)} \colon a \in \mS_2 \right\}
	= \sup \left\{ \frac{h_{K^\circ}(a)}{h_{L^\circ}(a)}, \frac{h_{L^\circ}(a)}{h_{K^\circ}(a)} \colon a \in \mS_2 \right\}
\end{equation}
and
\begin{equation}
\label{eq:up-mean_R0_ineq}
	R_0(\overline{p},\overline{q})
	\leq \sup \left\{ \frac{m_p(h_K(a),h_L(a))}{m_q(h_K(a),h_L(a))} \colon a \in \mS_2 \right\}.
\end{equation}
Using Lemma~\ref{lem:p-mean_basics}~(i) and Remark~\ref{rem:phi_pq}, this also shows that
\begin{align}
\label{eq:low-mean_R0_ineq}
\begin{split}
    R_0(\underline{p},\underline{q})
    & = R_0((\lowmean{q}{K}{L})^\circ, (\lowmean{p}{K}{L})^\circ)
    = R_0(\upmean{-q}{K^\circ}{L^\circ},\upmean{-p}{K^\circ}{L^\circ}) \\
	& \leq \sup \left\{ \frac{m_{-q}(h_{K^\circ}(a),h_{L^\circ}(a))}{m_{-p}(h_{K^\circ}(a),h_{L^\circ}(a))} : a \in \mS_2 \right\}
	= \sup \left\{ \frac{m_{p}(h_{K^\circ}(a),h_{L^\circ}(a))}{m_{q}(h_{K^\circ}(a),h_{L^\circ}(a))} : a \in \mS_2 \right\}.
\end{split}
\end{align}
Let us point out that
\eqref{eq:R0_p-means_relation}, \eqref{eq:up-mean_R0_ineq}, and \eqref{eq:low-mean_R0_ineq}
together provide an alternative proof for Theorem~\ref{thm:upper_bounds}~(ii).

For $p \geq q$, Remark~\ref{rem:phi_pq} and \eqref{eq:R0max_support_fct} show
\begin{equation}
\label{eq:R0_p-means_bound}
	\sup \left\{ \frac{m_p(h_K(a),h_L(a))}{m_q(h_K(a),h_L(a))} \colon a \in \mS_2 \right\}
	= \sup \left\{ \frac{m_{p}(h_{K^\circ}(a),h_{L^\circ}(a))}{m_{q}(h_{K^\circ}(a),h_{L^\circ}(a))} : a \in \mS_2 \right\}
	= \frac{m_p(R_0^{\max},1)}{m_q(R_0^{\max},1)}.
\end{equation}
Thus, \eqref{eq:R0_p-means_relation}, \eqref{eq:up-mean_R0_ineq}~--~\eqref{eq:R0_p-means_bound}
together yield
\[
	R_0(\underline{p},\overline{q})
		\leq \left\{ \begin{array}{c} R_0(\overline{p},\overline{q}) \\ R_0(\underline{p},\underline{q}) \end{array} \right\}
		\leq \frac{m_p(R_0^{\max},1)}{m_q(R_0^{\max},1)}.
\]
In particular, we obtain for general $p$ and $q$ from \eqref{eq:R0_infty_means} that
\[
	R_0^{\max}
		\leq R_0(\infty,\underline{p}) R_0(\underline{p},\overline{q}) R_0(\overline{q},-\infty)
		\leq \frac{m_\infty(R_0^{\max},1)}{m_p(R_0^{\max},1)} R_0(\underline{p},\overline{q}) \frac{m_q(R_0^{\max},1)}{m_{-\infty}(R_0^{\max},1)},
\]
which simplifies to
\[
	\frac{m_p(R_0^{\max},1)}{m_q(R_0^{\max},1)}
		\leq R_0(\underline{p},\overline{q}).
\]
This establishes Thereom~\ref{thm:upper_bounds}~(i) and, with \eqref{eq:R0_p-means_relation}, also the inequality in Theorem~\ref{thm:lower_bound}.

\pagebreak
We turn to the equality cases.
If $L = \lambda K$ for some $\lambda \geq 1$, then clearly $R_0^{\max} = \lambda$ and
for any $r \in [-\infty,\infty]$ also $\lowmean{r}{K}{L} = m_r(R_0^{\max},1) K = \upmean{r}{K}{L}$.
This implies tightness of the lower bound in Theorem~\ref{thm:lower_bound}.
If $\lambda < 1$, we reverse the roles of $K$ and $L$ and arrive at the same conclusion.
Equality in the lower bound for $n = 1$ is clear.

By $R_0(-\infty,\infty) \leq R_0(\underline{p},\overline{q}) \leq \min \{ R_0(\underline{p},\underline{q}), R_0(\overline{p},\overline{q}) \}$,
it suffices to construct for $n \geq 2$ and $R \geq 1$ some $K \in \CK^n_0$ with $R_0^{\max}(K,-K) = R$ such that
$R_0( K \cap (-K), \conv(K \cup (-K)) ) = 1$ to finish the proof.
To this end, let
\[
	K \coloneqq \B_{\infty} + \frac{R - 1}{R + 1} u^n
		= \bigcap_{i = 1}^{n-1} H_{(\pm u^i,1)}^\leq \cap H_{(u^n,\frac{2 R}{R + 1})}^\leq \cap H_{(-u^n,\frac{2}{R + 1})}^\leq.
\]
It is easy to see that indeed $R_0^{\max}(K,-K) = R$.
Moreover, $u^1 \in K \cap (-K) \cap \bd(\conv(K \cup (-K)))$ completes the proof.
\end{proof}

We turn to the proofs of the upper bounds on $R_0(\upmean{p}{K}{L},\lowmean{q}{K}{L})$ in Theorem~\ref{thm:upper_bounds}~(iii)~\m&~(iv).
They rely on different methods,
so we handle them separately, starting with (iii).
Let us point out that unlike for the other three combinations of types of means,
the upper bounds for this forth combination do not coincide with the lower bound in Theorem~\ref{thm:lower_bound}, unless 
$p = \infty$, or $q = -\infty$, or $R_0^{\max}(K,L) = 1$.
In the first two cases, we already know the exact value of $R_0(\upmean{p}{K}{L},\lowmean{q}{K}{L})$
due to Theorem~\ref{thm:same_p}~and~Theorem~\ref{thm:upper_bounds}~(i).
On the other hand, $R_0^{\max}(K,L) = 1$ is equivalent to $K = L$,
in which case all $p$-means trivially coincide.

\begin{proof}[Proof of Theorem~\ref{thm:upper_bounds}~(iii)]
We note that $\lowmean{q}{K}{L} \supset \lowmean{q}{1/R_0^{\max}L}{L} = m_q(1/R_0^{\max},1) L$,
and analogously $\lowmean{q}{K}{L} \allowbreak \supset m_q(1/R_0^{\max},1) K$.
Thus, by the properties of $p$-means of real numbers discussed in Remark~\ref{rem:mp},
\[
	\upmean{p}{K}{L} \subset \upmean{\infty}{K}{L} \subset m_{-q}(R_0^{\max},1) \lowmean{q}{K}{L}.
\]
Also $\upmean{p}{K}{L} \subset \upmean{p}{R_0^{\max} L}{L} = m_p(R_0^{\max},1) L$,
and analogously for $K$.
Hence,
\[
	\upmean{p}{K}{L}
		\subset m_p(R_0^{\max},1) \lowmean{-\infty}{K}{L}
		\subset m_p(R_0^{\max},1) \lowmean{q}{K}{L},
\]
which establishes the claimed upper bound.

We turn to the tightness of the bound.
If $p = \infty$ or $q = -\infty$,
then the upper bound coincides with the lower bound in Theorem~\ref{thm:lower_bound}.
If $p = -\infty$ or $q = \infty$, then the upper bound equals 1 and we can use
Theorem~\ref{thm:same_p} together with the tightness in Theorem~\ref{thm:upper_bounds}~(ii).
For finite $p$ and $q$, we first consider the case $(p,q) \in [1,\infty) \times [-1,\infty)$.
Let $n \geq 2$, $R \geq 1$, and define
$K \coloneqq \conv( \{ v, -R v \in \R^n : v_1, \ldots, v_{n-1} \in \{-1,1\}, v_n = 1 \} )$ (cf.~Figure~\ref{fig:R0_up-low-means}).

\begin{figure}[ht]
\def\scale{0.85}
\begin{tikzpicture}[scale=\scale, line join=round]
\def\p{2}
\def\q{1}
\def\x{((3^\p+1)/2)^(1/\p)}
\def\y{((3^\q+1)/2)^(1/\q)}

\draw[dotted] (-3,-3) -- (3,3);
\draw[dotted] (-3,3) -- (3,-3);

\draw[thick] (-3,-3) -- (3,-3) -- (1,1) -- (-1,1) -- cycle;

\draw[dashed,thick] (3,3) -- (-3,3);
\draw[dashed,thick] (-1,-1) -- (-3,3);
\draw[dashed,thick] (-1,-1) -- (1,-1);
\draw[dashed,thick] (1,-1) -- (3,3);

\fill (0,0) circle [radius=\ds];
\node[anchor=west] (0,0) {0};

\draw[orange,thick] ({\y},{\y}) -- ({\y},{-\y}) -- ({-\y},{-\y}) -- ({-\y},{\y}) -- cycle;

\draw[blue,thick,variable=\t,samples=100,domain={1/(3^(\p-1)+1)}:{3^(\p-1)/(3^(\p-1)+1)}]
		plot	(	{3 / (2^(1/\p) * (\t^(\p/(\p-1))+(1-\t)^(\p/(\p-1)))^((\p-1)/\p))},
				{(6 * \t - 3) / (2^(1/\p) * (\t^(\p/(\p-1))+(1-\t)^(\p/(\p-1)))^((\p-1)/\p))}	)
	--	({\x},{\x})
	--	({-\x},{\x})
	--	plot	(	{-3 / (2^(1/\p) * (\t^(\p/(\p-1))+(1-\t)^(\p/(\p-1)))^((\p-1)/\p))},
				{(3 - 6 * \t) / (2^(1/\p) * (\t^(\p/(\p-1))+(1-\t)^(\p/(\p-1)))^((\p-1)/\p))}	)
	--	({-\x},{-\x})
	--	({\x},{-\x})
	--	cycle;	
\end{tikzpicture}
\caption{
The bodies in the proof of Theorem~\ref{thm:upper_bounds}~(iii) for $n=2$, $p=2$, $q=1$, and $R=3$:
$K$ (black), $-K$ (dashed), $\upmean{p}{K}{-K}$ (blue), $\lowmean{q}{K}{-K}$~(orange).
}
\label{fig:R0_up-low-means}
\end{figure}

Clearly, $R_0^{\max}(K,-K) = R$.
Moreover, $R u^1$ lies in $\upmean{1}{K}{-K} \subset \upmean{p}{K}{-K}$.
If $F$ is a facet of $K$ or $-K$, then there exist $i \in [n]$ and $\sigma = \pm 1$ with
$\pos(F) = \pos( \{ v \in \{-1,1\}^n : v_i = \sigma \} )$.
Hence, by Theorem~\ref{thm:polytopes}, we obtain that
\begin{equation}
\label{eq:low-mean_construction}
	\lowmean{q}{K}{-K}
		= \conv \left( \left\{
        m_q(\|v\|_K^{-1}, \|v\|_{-K}^{-1}) v
        : v \in \{-1,1\}^n \right\} \right)
		= m_q(R,1) \B_\infty.
\end{equation}
Altogether, equality in the upper bound in this case follows from
\[
	R_0(\upmean{p}{K}{-K},\lowmean{q}{K}{-K})
		\geq \gauge{R u^1}_{\lowmean{q}{K}{-K}}
		= \frac{R}{m_q(R,1)} \gauge{u^1}_\infty
		= m_{-q}(R,1).
\]

If instead $(p,q) \in (-\infty,1] \times (-\infty,-1]$,
we obtain the claim from the first case applied to $K^\circ$.
Since $R_0(B^\circ,A^\circ) = R_0(A,B)$ for closed convex $A,B \subset \R^n$ containing $0$, we then have $R_0^{\max}(K,-K) = R_0^{\max}(K^\circ,-K^\circ)$, as well as from Lemma~\ref{lem:p-mean_basics}~(i)
\begin{align*}
    R_0(\upmean{p}{K}{-K},\lowmean{q}{K}{-K})
    & = R_0((\lowmean{q}{K}{-K})^\circ, (\upmean{p}{K}{-K})^\circ)
    \\
    & = R_0(\upmean{-q}{K^\circ}{-K^\circ},\lowmean{-p}{K^\circ}{-K^\circ}).
    \qedhere
\end{align*}
\end{proof}

The above proof shows that the upper bound in Theorem~\ref{thm:upper_bounds}~(iii) applies for any combination of $p$ and $q$.
However, its tightness is so far not considered
for $n \geq 2$, $R_0^{\max}(K,L) > 1$,
and either $(p,q) \in (-\infty,1) \times (-1,\infty)$ or $(p,q) \in (1,\infty) \times (-\infty,-1)$.
For the first case, we can, in fact, prove at least for $n = 2$ that the upper bound is not tight
(see Lemma~\ref{lem:R0-up-low-means_strict} below).
For the second case, we show in Remark~\ref{rem:R0-up-low-means_ineq_2} 
that the additional upper bound from Theorem~\ref{thm:upper_bounds}~(iv) is strictly smaller.
Unfortunately, we do not have a more explicit representation of the bound if $\frac{p}{p-1} < -q$.
Even though we can reduce the computation of the tight upper bound to what is essentially a one-dimensional maximization problem,
finding its critical points for general $p$ and $q$ becomes a seemingly untractable problem.
Note that the easily identifiable critical points show that the bound is
at least $\frac{R_0^{\max}(K,L) + 1}{2}$.

We now continue with the proof of Theorem~\ref{thm:upper_bounds}~(iv).
It is based on a different representation of upper $p$-means for $p > 1$ given in \cite{LuYaZh}.

\begin{proof}[Proof of Theorem~\ref{thm:upper_bounds}~(iv)]
For $p > 1$, it is shown in \cite{LuYaZh} that
\begin{equation}
\label{eq:up-mean_alt_rep}
	\upmean{p}{K}{L}
		= 2^{- \frac{1}{p}} \left\{ \mu^{\frac{p-1}{p}} x + (1 - \mu)^{\frac{p-1}{p}} y : x \in K, y \in L, \mu \in [0,1] \right\}.
\end{equation}
By $q \in (-\infty,-1)$,
the claimed inequality follows from
\begin{align*}
    & R_0(\upmean{p}{K}{L}, \lowmean{q}{K}{L})
    = \sup \left\{
        \gauge{z}_{\lowmean{q}{K}{L}}: z \in \upmean{p}{K}{L}
    \right\} \\
	& = 2^{- \frac{1}{p}} \sup \left\{
		m_{-q} \left(
            \gauge{\mu^{\frac{p-1}{p}} x + (1 - \mu)^{\frac{p-1}{p}} y}_K,
            \gauge{\mu^{\frac{p-1}{p}} x + (1 - \mu)^{\frac{p-1}{p}} y}_L
        \right) : x \in K, y \in L, \mu \in [0,1]
    \right\} \\
	& \leq 2^{- \frac{1}{p}} \sup \left\{
        m_{-q} \left(
            \mu^{\frac{p-1}{p}} + (1 - \mu)^{\frac{p-1}{p}} R_0^{\max},
			\mu^{\frac{p-1}{p}} R_0^{\max} + (1 - \mu)^{\frac{p-1}{p}}
        \right) : \mu \in [0,1]
    \right\} \\
	& = 2^{\frac{1}{q} - \frac{1}{p}} \sup \left\{
        \gauge{
            R_0^{\max} \cdot \begin{psmallmatrix} (1-\mu)^{\frac{p-1}{p}} \\ \mu^{\frac{p-1}{p}} \end{psmallmatrix}
			+ \begin{psmallmatrix} \mu^{\frac{p-1}{p}} \\ (1-\mu)^{\frac{p-1}{p}} \end{psmallmatrix}
        }_{-q} : \mu \in [0,1]
    \right\} \\
	& = 2^{\frac{1}{q} - \frac{1}{p}} \sup \left\{
        \gauge{
            \begin{pmatrix} R_0^{\max} & 1 \\ 1 & R_0^{\max} \end{pmatrix} w
        }_{-q} : w \in \R^2, \gauge{w}_{\frac{p}{p-1}} \leq 1
    \right\}.
\end{align*}

When $\frac{p}{p-1} \geq -q$, we have for all $w \in \R^2$ that
\[
	\gauge{w}_{-q}
		= 2^{-\frac{1}{q}} m_{-q}(\vert w_1 \vert,\vert w_2 \vert)
		\leq 2^{-\frac{1}{q}} m_{\frac{p}{p-1}}(\vert w_1 \vert,\vert w_2 \vert)
		= 2^{-\frac{1}{q} - \frac{p-1}{p}} \gauge{w}_{\frac{p}{p-1}},
\]
with equality if $\vert w_1 \vert = \vert w_2 \vert$.
Thus, we obtain for $w \in \R^2$ with $\gauge{w}_{\frac{p}{p-1}} \leq 1$ that
\[
	2^{\frac{1}{q} - \frac{1}{p}} \gauge{\begin{pmatrix} R_0^{\max} & 1 \\ 1 & R_0^{\max} \end{pmatrix} w}_{-q}
		\leq 2^{\frac{1}{q} - \frac{1}{p}} (R_0^{\max} \gauge{w}_{-q} + \gauge{w}_{-q})
		\leq \frac{R_0^{\max} + 1}{2},
\]
with equality if $w_1 = w_2$ and $\gauge{w}_{\frac{p}{p-1}} = 1$.
This verifies the explicit representation of the upper bound in this case.

For the tightness of the inequality,
let $K \coloneqq \conv( \{ v, -R v \in \R^n : v_1, \ldots, v_{n-1} \in \{-1,1\}, v_n = 1 \} )$,
which clearly satisfies $R_0^{\max}(K,-K) = R$.
Moreover, if we choose $x \coloneqq R (u^1 - u^n) \in K$ and $y \coloneqq R (u^1 + u^n) \in -K$,
then $\mu^{\frac{p-1}{p}} x + (1 - \mu)^{\frac{p-1}{p}} y \allowbreak \in F_K \cap F_{-K}$ for every $\mu \in [0,1]$
with $F_K$ and $F_{-K}$ being the positive hulls of the facets of $K$ and $-K$ that both intersect the rays $[0,\infty) x$ and $[0,\infty) y$.
Since $\gauge{x}_{-K} = \gauge{y}_K = R$, we obtain
$\gauge{\mu^{\frac{p-1}{p}} x + (1 - \mu)^{\frac{p-1}{p}} y}_K = \mu^{\frac{p-1}{p}} + (1 - \mu)^{\frac{p-1}{p}} R$
and $\gauge{\mu^{\frac{p-1}{p}} x + (1 - \mu)^{\frac{p-1}{p}} y}_{-K} = \mu^{\frac{p-1}{p}} R + (1 - \mu)^{\frac{p-1}{p}}$
for all $\mu \in [0,1]$.
Hence, equality is attained in the above estimate for the choice $L = -K$.
\end{proof}

\begin{remark}
\label{rem:R0-up-low-means_ineq_2}
In the above setting, the maps
\[
	[0,1] \ni \lambda \mapsto \gauge{\begin{pmatrix} R_0^{\max} & 1 \\ 1 & R_0^{\max} \end{pmatrix} \begin{pmatrix} 1 - \lambda \\ \lambda \end{pmatrix}}_{-q}
		\quad \text{and} \quad
	[0,1] \ni \lambda \mapsto \gauge{\begin{pmatrix} 1 - \lambda \\ \lambda \end{pmatrix}}_{\frac{p}{p-1}}
\]
are both strictly convex.
We obtain for $\lambda \in [0,1]$
\[
	\gauge{\begin{pmatrix} R_0^{\max} & 1 \\ 1 & R_0^{\max} \end{pmatrix} \begin{pmatrix} 1 - \lambda \\ \lambda \end{pmatrix}}_{-q}
		\leq \gauge{\begin{pmatrix} R_0^{\max} & 1 \\ 1 & R_0^{\max} \end{pmatrix} \begin{pmatrix} 1 \\ 0 \end{pmatrix}}_{-q}
		= 2^{-\frac{1}{q}} m_{-q}(R_0^{\max},1)
\]
with equality if and only if $\lambda \in \{ 0,1 \}$, and
\[
	\gauge{\begin{pmatrix} 1 - \lambda \\ \lambda \end{pmatrix}}_{\frac{p}{p-1}}
		\geq \gauge{\begin{pmatrix} 1/2 \\ 1/2 \end{pmatrix}}_{\frac{p}{p-1}}
		= 2^{-\frac{1}{p}}
\]
with equality if and only if $\lambda = 1/2$.

Altogether, the proof of Theorem~\ref{thm:upper_bounds}~(iv) shows
\begin{align*}
	R_0(\upmean{p}{K}{L},\lowmean{q}{K}{L})
    & \leq 2^{\frac{1}{q} - \frac{1}{p}} \sup \left\{
        \frac{
            \gauge{
                \begin{psmallmatrix} R_0^{\max} & 1 \\ 1 & R_0^{\max} \end{psmallmatrix}
                \begin{psmallmatrix} 1 - \lambda \\ \lambda \end{psmallmatrix}
            }_{-q}
        }{
            \gauge{\begin{psmallmatrix} 1 - \lambda \\ \lambda \end{psmallmatrix}}_{\frac{p}{p-1}}
        } : \lambda \in [0,1]
    \right\} \\
	& < 2^{\frac{1}{q} - \frac{1}{p}} \frac{2^{-\frac{1}{q}} m_{-q}(R_0^{\max},1) }{2^{-\frac{1}{p}}}
	= m_{-q}(R_0^{\max},1)
\intertext{
and thus, with $R_0^{\max} = R_0^{\max}(K,L) = R_0^{\max}(K^\circ,L^\circ)$ and Lemma~\ref{lem:p-mean_basics}~(i), also
}
	R_0(\upmean{p}{K}{L},\lowmean{q}{K}{L})
	& = R_0((\lowmean{q}{K}{L})^\circ,(\upmean{p}{K}{L})^\circ)
    = R_0(\upmean{-q}{K^\circ}{L^\circ},\lowmean{-p}{K^\circ}{L^\circ})
    \\
	& < m_{p}(R_0^{\max},1).
\end{align*}
\end{remark}

We turn to our last result about the tightness of Theorem~\ref{thm:upper_bounds}~(iii) discussed above.

\begin{lemma}
\label{lem:R0-up-low-means_strict}
Let $n = 2$, $K, L \in \CK^2_0$ with $R_0^{\max}(K,L) > 1$,
and $(p,q) \in (-\infty,1) \times (-1,\infty)$.
Then
\[
	R_0(\upmean{p}{K}{L},\lowmean{q}{K}{L})
		< \min \{ m_p(R_0^{\max}(K,L),1), m_{-q}(R_0^{\max}(K,L),1) \}.
\]
\end{lemma} \begin{proof}
In the proof of Theorem~\ref{thm:upper_bounds}~(iii), it is shown that
\begin{equation}
\label{eq:chain_strict}
	\upmean{p}{K}{L} \subset \conv(K \cup L) \subset m_{-q}(R_0^{\max},1) \lowmean{q}{K}{L}.
\end{equation}
Now, assume we would have $R_0(\upmean{p}{K}{L},\lowmean{q}{K}{L}) = m_{-q}(R_0^{\max},1)$.
Then there exists a common boundary point $x$ of $\upmean{p}{K}{L}$ and $m_{-q}(R_0^{\max},1) \lowmean{q}{K}{L}$,
which by \eqref{eq:chain_strict} must also be a boundary point of $\conv(K \cup L)$.
Moreover, there exists some vector $a \in \mS_2$ such that $H_{(a,a^T x)}^\leq$ supports all three bodies at $x$.
Our goal is to contradict $x \in \upmean{p}{K}{L}$ by finding an appropriate $\tilde{a} \in \R^2$
with $\tilde{a}^T x > h_{\upmean{p}{K}{L}}(\tilde{a})$.
With $R_0^{\max} = R_0^{\max}(K,L) = R_0^{\max}(K^\circ,L^\circ)$ and Lemma~\ref{lem:p-mean_basics}~(i),
the strict inequality with respect to $m_{-q}(R_0^{\max},1)$ then also implies
\begin{align*}
	R_0(\upmean{p}{K}{L},\lowmean{q}{K}{L})
    & = R_0((\lowmean{q}{K}{L})^\circ,(\upmean{p}{K}{L})^\circ)
	= R_0(\upmean{-q}{K^\circ}{L^\circ},\lowmean{-p}{K^\circ}{L^\circ})
    \\
	& < m_p(R_0^{\max},1).
\end{align*}

Since $p < 1$, Lemma~\ref{lem:common_boundary_different_p}~(ii) implies $h_K(a) = h_L(a) = a^T x$,
and from $\upmean{p}{K}{L} \subset \upmean{1}{K}{L} = \frac{K+L}{2}$,
we obtain that there exist $v \in K$ and $w \in L$ such that $x = \frac{v + w}{2}$ and $a^T v = h_K(a) = h_L(a) = a^T w$.
Now, $[v,w] \subset \conv(K \cup L) \cap H_{(a,a^T x)}$ implies by our assumptions that
$[v,w] \subset \bd(m_{-q}(R_0^{\max},1) \lowmean{q}{K}{L})$.

For $t \in [0,1]$, define $\gamma(t) \coloneqq (1-t) v + t w$.
Then $\left[ \frac{1}{R_0^{\max}} v,w \right] \subset L$ yields
\[
	\frac{\gamma(t)}{(1 - t) R_0^{\max} + t}
		= \frac{(1 - t) R_0^{\max}}{(1 - t) R_0^{\max} + t} \left( \frac{1}{R_0^{\max}} v \right) + \frac{t}{(1 - t) R_0^{\max} + t} w \in L
\]
and thus $\gauge{\gamma(t)}_L \leq (1 - t) R_0^{\max} + t$.
From $[v,w] \subset \bd(m_{-q}(R_0^{\max},1) \lowmean{q}{K}{L})$ we derive
\begin{equation}
\label{eq:gamma_gauge}
	m_{-q}(R_0^{\max},1)
		= \gauge{\gamma(t)}_{\lowmean{q}{K}{L}}
		\leq m_{-q}(\gauge{\gamma(t)}_K,\gauge{\gamma(t)}_L)
		\leq m_{-q}(\gauge{\gamma(t)}_K,(1-t) R_0^{\max} + t).
\end{equation}
Now, the function
\[
	g_q(t) \coloneqq		\begin{cases}
						((R_0^{\max})^{-q} + 1 - ((1 - t) R_0^{\max} + t)^{-q})^{-\frac{1}{q}},	& \text{if } q \neq 0, \\
						\frac{R_0^{\max}}{(1 - t) R_0^{\max} + t},								& \text{if } q = 0,
					\end{cases}
\]
satisfies
\begin{equation}
\label{eq:g_q_prop}
	m_{-q}(g_q(t),(1-t) R_0^{\max} + t) = m_{-q}(R_0^{\max},1).
\end{equation}
Since $m_{-q}$ is strictly increasing in its two arguments for finite $q$,
\eqref{eq:gamma_gauge} implies $\gauge{\gamma(t)}_K \geq g_q(t)$ (cf.~Figure~\ref{fig:R0-up-low-means_strict}).

\begin{figure}[ht]
\def\scale{1.3}
\begin{tikzpicture}[scale=\scale, line join=round]
\def\R{3}
\def\p{(-2)}
\def\q{1}

\fill[orange!20] (-2,0) -- (0,0) -- (0,1) -- (-2,{1+2/\R});
\fill[orange!20] (0,-2) -- (0,0) -- (\R,0) -- ({\R+2/\R},-2);
\fill[orange!20] (-2,-2) -- (0,-2) -- (0,0) -- (-2,0);
\fill[orange!20] (0,0) -- plot[variable=\t,domain=0:1,smooth]	(	{(1-\t)*\R/((\R^(-\q) + 1 - (\R - \t*(\R-1))^(-\q))^(-1/\q))},
												{\t*\R/((\R^(-\q) + 1 - (\R - \t*(\R-1))^(-\q))^(-1/\q))}	);

\draw[-latex,thick] (-2,0) -- ({\R+1},0);
\draw[-latex,thick] (0,-2) -- (0,{\R+1});

\draw[red,dotted] (\R,0) -- (0,1);
\draw[blue,dotted] (0,\R) -- (1,0);

\draw[blue,thick] (1,0) -- ({1+2/\R},-2);
\draw[blue,thick] (0,\R) -- (-2,{\R+2/\R});

\draw[dashdotted,thick] ({((\R^\q+1)/2)^(1/\q)},0) -- (0,{((\R^\q+1)/2)^(1/\q)});

\draw[orange,thick] (0,1) -- (-2,{1+2/\R});
\draw[orange,thick] (\R,0) -- ({\R+2/\R},-2);
\draw[orange,thick,variable=\t,domain=0:1,smooth] plot	(	{(1-\t)*\R/((\R^(-\q) + 1 - (\R - \t*(\R-1))^(-\q))^(-1/\q))},
												{\t*\R/((\R^(-\q) + 1 - (\R - \t*(\R-1))^(-\q))^(-1/\q))}	);

\draw[dashdotted,thick] ({\R+1},-1) -- (-1,{\R+1});

\draw[dashed,red,thick] ({\R+1},{-1/(1 + (\R-1)/(\R^(\q+1)))}) -- ({\R - (\R + 1) * (1 + (\R-1)/(\R^(\q+1))))},{\R+1});
\draw[dashed,blue,thick] ({- (1 + (\R-1)/(\R^(\q+1)))},{\R+1}) -- ({\R+1},{((\R-1)/(\R^\q) - 1)/(1 + (\R-1)/(\R^(\q+1)))});

\draw[-latex] ({\R-1},1) -- (\R,2);
\node[anchor=west] at (\R,2) {$a$};

\draw[-latex,red] (1,{(\R-1)/(1 + (\R-1)/\R^(\q+1))}) -- (2,{(\R-1)/(1 + (\R-1)/\R^(\q+1)) + 1 + (\R-1)/(\R^(\q+1))});
\node[anchor=west] at (2,{(\R-1)/(1 + (\R-1)/\R^(\q+1)) + 1 + (\R-1)/(\R^(\q+1))}) {$a_q$};

\fill (\R,0) circle [radius=\ds] node[anchor=south west] {$v$};
\fill ({((\R^\q+1)/2)^(1/\q)},0) circle [radius=\ds] node[anchor=north] {$\frac{v}{m_{-q}(R_0^{\max},1)}$};
\fill (1,0) circle [radius=\ds] node[anchor=north east] {$\frac{v}{R_0^{\max}}$};

\fill (0,\R) circle [radius=\ds] node[anchor=south west] {$w$};
\fill (0,{((\R^\q+1)/2)^(1/\q)}) circle [radius=\ds] node[anchor=east] {$\frac{w}{m_{-q}(R_0^{\max},1)}$};
\fill (0,1) circle [radius=\ds] node[anchor=north east] {$\frac{w}{R_0^{\max}}$};

\fill ({\R/2},{\R/2}) circle [radius=\ds] node[anchor=west] {$x$};
\end{tikzpicture}
\caption{
The situation in the proof of Lemma~\ref{lem:R0-up-low-means_strict}:
None of the points on the orange curve (which consist of $\gamma/g_q$ and two rays) lie in $\inte(K)$, so $K$ is a subset of the orange area.
The two dashed lines are parallel and support $K$ at $v$ and $L$ at $w$, respectively.
The dash-dotted line and segment are also parallel.
The line containing $x$ supports $\conv( K \cup L)$ at $x$,
whereas the segment belongs to $\bd(\lowmean{q}{K}{L})$.
}
\label{fig:R0-up-low-means_strict}
\end{figure}

In particular, we obtain $\gauge{w}_K = \gauge{\gamma(1)}_K \geq g_q(1) = R_0^{\max}$
and since $w \in L \subset R_0^{\max} K$, it follows $\frac{1}{R_0^{\max}} w \in \bd(K)$.
Now, $v \in K$ and $R_0^{\max} > 1$ show $w \neq v$.
Since $[v,w] \subset \bd(\conv(K \cup L))$ and $0 \in \inte(\conv(K \cup L))$,
the line $\aff(\{v,w\})$ cannot contain $0$, i.e., $v$ and $w$ are linearly independent.
Using Lemma~\ref{lem:p-mean_basics}~(iv), we may w.l.o.g.~assume
that $v$ and $w$ are orthogonal and $\gauge{v}_2 = \gauge{w}_2 = 1$ (applying an appropriate linear transformation otherwise).

Next, we define $a_q \coloneqq v + ( 1 + \frac{R_0^{\max} - 1}{(R_0^{\max})^{q+1}} ) w$
and claim for all $t \in [0,1]$ that $a_q^T \frac{\gamma(t)}{g_q(t)} \leq 1$, or equivalently $a_q^T \gamma(t) \leq g_q(t)$.
Due to the strict monotonicity of $m_{-q}$, it suffices to show
\begin{equation}
\label{eq:a_q_ineq}
	m_{-q}(a_q^T \gamma(t),(1-t) R_0^{\max} + t)
		\leq m_{-q}(g_q(t),(1-t) R_0^{\max} + t).
\end{equation}
By \eqref{eq:g_q_prop}, the orthonormality of $\{v,w\}$, and the definitions of $a_q$ and $\gamma(t)$, this is equivalent to
\[
	f(t) \coloneqq
		m_{-q} \left( 1 + t \frac{R_0^{\max} - 1}{(R_0^{\max})^{q+1}}, R_0^{\max} - t (R_0^{\max} - 1) \right) \\
		\leq m_{-q}(R_0^{\max},1).
\]
There exists some $\varepsilon > 0$ such that $(-\varepsilon,1+\varepsilon) \ni t \mapsto f(t)$ is a well-defined function.
Since $-q < 1$, it follows from the reverse Minkowski inequality that $f$ is concave.
Moreover, $f$ is differentiable and it is straightforward to compute $f'(0) = 0$.
Hence, $f$ attains its maximum over $t \in [0,1]$ at $t = 0$ with $f(0) = m_{-q}(R_0^{\max},1)$.
In summary, \eqref{eq:a_q_ineq}, and thus $a_q^T \gamma(t) \leq g_q(t)$, is proven.
Together with $\gauge{\gamma(t)}_K \geq g_q(t)$, we obtain for every $y \in K \cap \pos(\{ v,w \})$
\[
	a_q^T y
	\leq \sup \left\{ \frac{a_q^T \gamma(t)}{\gauge{\gamma(t)}_K} : t \in [0,1] \right\}
	\leq \sup \left\{ \frac{a_q^T \gamma(t)}{g_q(t)} : t \in [0,1] \right\}
	\leq 1,
\]
with equality from left to right if $y = v$.

Since $\frac{1}{R_0^{\max}} w \in \bd(K)$, $v \in K$, and $0 \in \inte(K)$
we must have
\begin{align*}
	K \cap \pos(\{-v,w\})
		& \subset \left( v + \pos \left(\left\{ -v, \left( \frac{1}{R_0^{\max}} w - v \right) \right\}\right) \right) \cap \pos(\{ -v,w \}) \\
		& = \left[ 0, \frac{1}{R_0^{\max}} w \right] + \pos \left(\left\{ -v, \left( \frac{1}{R_0^{\max}} w - v \right) \right\}\right)
\end{align*}
(cf.~Figure~\ref{fig:R0-up-low-means_strict}).
From $a_q^T v = 1 \geq a_q^T \frac{\gamma(1)}{g_q(1)} = a_q^T ( \frac{1}{R_0^{\max}} w )$ and $a_q^T (-v) < 0$,
we obtain for every $y \in K \cap \pos(\{-v,w\})$ that $a_q^T y \leq 1$.
Exchanging the roles of $(K,v)$ and $(L,w)$ further shows
\begin{align*}
	L \cap \pos(\{v,-w\})
		& \subset \left[ 0, \frac{1}{R_0^{\max}} v \right] + \pos \left(\left\{ -w, \left( \frac{1}{R_0^{\max}} v - w \right) \right\}\right),
\intertext{
so $K \subset R_0^{\max} L$ yields
}
	K \cap \pos(\{v,-w\})
		& \subset \left[ 0, v \right] + \pos \left(\left\{ -w, \left( \frac{1}{R_0^{\max}} v - w \right) \right\}\right).
\end{align*}
Using the orthonormality of $\{v,w\}$, we compute
$a_q^T (\frac{1}{R_0^{\max}} v - w) = \frac{1}{R_0^{\max}} - 1 - \frac{R_0^{\max} - 1}{(R_0^{\max})^{q+1}} < 0$ and $a_q^T (-w) < 0$,
which shows for all $y \in K \cap \pos(\{v,-w\})$ that $a_q^T y \leq a_q^T v = 1$.
Finally, every $y \in K \cap \pos(\{-v,-w\})$ clearly satisfies $a_q^T y \leq 0$.
Altogether, $H_{(a_q,1)}^\leq$ supports $K$ at $v$ (cf.~Figure~\ref{fig:R0-up-low-means_strict}).

Next, we argue that $H_{(a_q,a_q^T w)}^\leq$ supports $L$ at $w$.
Since $L \subset \conv(K \cup L)$, we have
\[
	L \cap \pos(\{v,w\})
		\subset \conv(K \cup L) \cap \pos(\{v,w\})
		= \conv(\{v,w,0\}).
\]
Among the three points in the convex hull, $w$ uniquely has the largest inner product with $a_q$ by $R_0^{\max} > 1$.
The intersections of $L$ with $\pos(\{v,-w\})$, $\pos(\{-v,w\})$, and $\pos(\{-v,-w\})$ are handled analogously as for $K$
(cf.~Figure~\ref{fig:R0-up-low-means_strict}).
Altogether, it follows $h_K(a_q) = a_q^T v < a_q^T w = h_L(a_q)$.
In particular, $p < 1$ implies
\[
	h_{\upmean{p}{K}{L}}(a_q) \leq m_p(h_K(a_q),h_L(a_q)) < m_1(h_K(a_q),h_L(a_q)) = a_q^T \left( \frac{v + w}{2} \right) = a_q^T x,
\]
contradicting $x \in \upmean{p}{K}{L}$ as desired.
\end{proof}

Let us point out that almost all steps in the above proof apply for any dimension $n \geq 2$.
The only time the assumption $n = 2$ is needed is when we conclude from $\sup\{ a_q^T y : y \in K \cap \lin(\{v,w\}) \} = a_q^T v$
that $H_{(a_q,a_q^T v)}^\leq$ contains $K$ (and for the analogous statement with $(L,w)$ instead of $(K,v)$).

Before we end this section with the proof of Theorem~\ref{thm:stability},
let us collect some remarks about the inequalities on the covering radii between $p$-means.
They concern the tightness of the upper bounds on $R_0^{\max}(\upmean{p}{K}{L},\lowmean{q}{K}{L})$,
an extension of our results to bodies not necessarily containing $0$ in their interior,
and a comparison to the results in \cite{BrDiMe}.

\begin{remark}
For $n \geq 2$ and $R \geq 1$, let $K \coloneqq \conv( \{ v, -R v \in \R^n : v_1, \ldots, v_{n-1} \in \{-1,1\}, v_n = 1 \} )$
be the body considered in the proofs of tightness for
Theorem~\ref{thm:upper_bounds}~(iii)~\m&~(iv).
There, it is shown that $R_0(\upmean{p}{K}{-K},\lowmean{q}{K}{-K})$ reaches the upper bound in the respective inequality
if $(p,q) \in (1,\infty) \times \R$, which can be extended to $(p,q) \in [1,\infty] \times [-\infty,\infty]$ by continuity.
This can be further extended to $(p,q) \in [-\infty,1] \times [-\infty,-1]$ by the following argument.

A somewhat lengthy, but not difficult computation involving Theorem~\ref{thm:polytopes}
shows for $p \in [-\infty,1]$ that all facets of the polytope $\upmean{p}{K}{-K}$
are parallel to appropriately chosen facets of $K$ or $-K$.
From this, we obtain
\[
	\upmean{p}{K}{-K} = H_{(\pm u^n,m_p(R,1))}^\leq
		\cap \bigcap_{i \in [n-1]} H_{(\pm \frac{R+1}{2 R} u^i \pm \frac{R-1}{2 R} u^n,m_p(R,1))}^\leq,
\]
where all signs are taken independently.
Using this and $\gauge{\cdot}_{\lowmean{q}{K}{-K}} = m_{-q}(\gauge{\cdot}_K,\gauge{\cdot}_{-K})$ for $q \in [-\infty,-1]$,
it can be shown that $\upmean{p}{K}{-K}$ and $m_p(R,1) \lowmean{q}{K}{-K}$ have $\frac{R m_p(R,1)}{m_1(R,1)} u^1$ as a common boundary point 
for the given ranges of $p$ and $q$.
In particular, the results in this section show that this construction yields tightness in the upper bound on $R_0(\upmean{p}{K}{L},\lowmean{q}{K}{L})$
for $n \geq 2$ and all triples
$(p,q,R) \notin [-\infty,1) \times (-1,\infty] \times (1,\infty)$.

One might expect that the largest possible value of $R_0(\upmean{p}{K}{L},\lowmean{q}{K}{L})$
is achieved by the above construction for this constellation as well.
For $q \in [-1,\infty]$, we already obtained in the proof of Theorem~\ref{thm:upper_bounds}~(iii) 
that $\lowmean{q}{K}{-K} = m_q(R,1) \B_{\infty}$ (cf.~\eqref{eq:low-mean_construction}).
Hence, we can compute in this case
\begin{equation}
\label{eq:R0_open_example}
	R_0^{\max}(\upmean{p}{K}{-K},\lowmean{q}{K}{-K}) = \frac{m_p(R,1) \, m_{-q}(R,1)}{m_1(R,1)}.
\end{equation}
For $R > 1$ and $(p,q) \in (\{-\infty\} \times (-1,\infty)) \cup ((-\infty,1) \times \{\infty\})$,
this is strictly less than the known optimal upper bound $1$
from Theorem~\ref{thm:same_p} and Theorem~\ref{thm:upper_bounds}~(ii).
If we assume the optimal upper bound to be continuous in $(p,q)$,
then it is larger than the value in \eqref{eq:R0_open_example}
at least for $p$ sufficiently close to $-\infty$ or $q$ sufficiently close to $\infty$, as well.
\end{remark}

\begin{remark}
\label{rem:R0_boundary}
Theorems~\ref{thm:lower_bound}~and~\ref{thm:upper_bounds} remain valid
if we only require $K,L \in \CK^n$ with $0 \in K \cap L$ and $R_0^{\max}(K,L) \in (0,\infty)$,
instead of $K,L \in \CK^n_0$.
Surely, we need to extend the definitions of upper and lower $p$-means appropriately in this case.
We first observe that $R_0^{\max}(K,L) < \infty$ implies for $v,a \in \mS_2$
that $\gauge{v}_K = \infty$ if and only if $\gauge{v}_L = \infty$,
and $h_K(a) = 0$ if and only if $h_L(a) = 0$.
Thus, with the conventions $m_p(0,0) = 0$, $m_p(\infty,\infty) = \infty$, and $\frac{1}{\infty} = 0$ for any $p \in [-\infty,\infty]$,
one defines $\upmean{p}{K}{L}$ and $\lowmean{p}{K}{L}$ as in Definition~\ref{def:p-means}.

Now, to see that Theorems~\ref{thm:lower_bound}~and~\ref{thm:upper_bounds} remain valid,
we first observe that Lemma~\ref{lem:p-mean_basics}~(ii)~and~(iii) remain true with almost identical proofs.
Moreover, $R_0^{\max}(K,L) \in (0,\infty)$ shows $d \coloneqq \dim(K) = \dim(L) \geq 1$,
so comparing $d$-dimensional volumes yields $R_0^{\max}(K,L) \geq 1$.
It also keeps true
\[
	\lowmean{-\infty}{K}{L} = \upmean{-\infty}{K}{L} = K \cap L
		\quad \text{and} \quad
	\lowmean{\infty}{K}{L} = \upmean{\infty}{K}{L} = \conv(K \cup L).
\]
Furthermore, it can be shown that $\lowmean{p}{K}{L}$ is closed for any $p \in [-\infty,\infty]$
by first proving that the set $V$ generating $\lowmean{p}{K}{L}$ in Definition~\ref{def:p-means}
satisfies $\cl(V) \subset \conv(V)$ under the above assumptions.
Thus, all steps in the proof of Theorem~\ref{thm:lower_bound} and Theorem~\ref{thm:upper_bounds}~(i)~--~(iii)
can still be executed, where the only step we need to be more careful with is the proof of \eqref{eq:R0_p-means_bound}.
Here, we use that $R_0^{\max}(K,L) > 0$ shows that with \eqref{eq:frac_def},
we may disregard all $a \in \mS_2$ with $h_K(a) = h_L(a) = 0$ or $h_{K^\circ}(a) = h_{L^\circ}(a) = \infty$.
Then we can again apply Remark~\ref{rem:phi_pq}.

Lastly, we observe that the proof of \eqref{eq:up-mean_alt_rep} in \textup{\cite{LuYaZh}}
also applies for $K,L \in \CK^n$ with $0 \in K \cap L$ instead of $K,L \in \CK^n_0$.
Hence, the given proof of Theorem~\ref{thm:upper_bounds}~(iv) can still be executed as well,
where we replace $\gauge{\cdot}_{\lowmean{q}{K}{L}} = m_{-q}(\gauge{\cdot}_K,\gauge{\cdot}_L)$ in the main estimate
with the always
true inequality $\gauge{\cdot}_{\lowmean{q}{K}{L}} \leq m_{-q}(\gauge{\cdot}_K,\gauge{\cdot}_L)$
to ensure its correctness in the more general setting.
\end{remark}

\begin{remark}
\label{rem:BDM_comparison}
Let us compare the bounds obtained in this section to the results in \textup{\cite{BrDiMe}}.
There, the circumradii between the four standard symmetrizations of Minkowski centered convex bodies $K \in \CK^n_0$ are considered.
Since circumradii and covering radii coincide when only $0$-symmetric bodies are involved,
the results in \textup{\cite{BrDiMe}} could also be expressed using the latter. Now,
\textup{\cite[Theorem~1.3]{BrDiMe}} shows upper bounds on these radii,
which can be obtained as special cases from Theorem~\ref{thm:upper_bounds}~(i)~\m&~(iii) here.
However, the lower bounds specified in \textup{\cite[Theorem~1.7]{BrDiMe}} are stronger than those obtained here,
which heavily relies on the restriction to $L = -K$.
Moreover, \textup{\cite[Theorem~1.5]{BrDiMe}} also states that the upper bound on the radii
between some of the four standard symmetrizations in forward direction of $p$-values can be improved
for $n$ even and sufficiently large values of the Minkowski asymmetry (namely, larger than the asymmetry threshold of means \eqref{eq:threshold}).
This contrasts the tightness results obtained here,
which show that,
even under the restriction to $L = -K$,
there is no qualitative change in the optimal upper bound for large values of $R_0^{\max}(K,-K)$
if we allow $K$ to be positioned arbitrarily.
This highlights that the optimality in $-K \subset^{opt} s(K) K$ for a Minkowski centered $K \in \CK^n_0$
is essential for \textup{\cite[Theorem~1.5]{BrDiMe}}.
\end{remark}

Finally, we close this section with the proof of Theorem~\ref{thm:stability}.

\begin{proof}[Proof of Theorem~\ref{thm:stability}]
By Theorem~\ref{thm:lower_bound}, we have
\[
    \Phi_{p,q}(R_0^{\max}(K,L))
    \leq R_0(M_q(K,L),M_p(K,L))
    \leq R_0^{\max}(M_p(K,L),M_q(K,L))
    = 1 + \varepsilon,
\]
with equality throughout if $K$ and $L$ are dilates of each other.
For general $K$ and $L$, equality still holds 
throughout if $p=-\infty$ or $q=\infty$ according to Theorem~\ref{thm:same_p} and Theorem~\ref{thm:upper_bounds}~(i)~\m&~(ii).
Now, Remark~\ref{rem:phi_pq} shows for $p < q$ that
$\Phi_{p,q}'(t) > 0$ for all $t > 1$,
so the restriction of $\Phi_{p,q}$ to the domain $[1,\infty)$ is invertible with strictly increasing inverse function.
It follows
\[
    R_0^{\max}(K,L)
    \leq \Phi_{p,q}^{-1}(1+\varepsilon),
\]
with equality cases corresponding to the above.
Therefore, the theorem can be established by proving appropriate upper and lower bounds to $\Phi_{p,q}^{-1}(1+\varepsilon)$.

For (i), in which case $p,q \in \R$, we define for $\delta \in [0,1)$ the function
\[
    g_{p,q,\delta} \colon [1,\infty) \to [1,\infty),
    g_{p,q,\delta}(x) = 1 + \frac{q-p}{8} (1-\delta) (x-1)^2.
\]
This function is invertible with strictly increasing inverse
\[
    g_{p,q,\delta}^{-1} \colon [1,\infty) \to [1,\infty),
    g_{p,q,\delta}^{-1}(y)
    = 1 + \sqrt{ \frac{8 (y-1) }{(1-\delta) (q-p)} }.
\]
Lemma~\ref{lem:stability} shows for $y \in \Phi_{p,q}([1,\infty))$ that
\[
    1 + \sqrt{ \frac{8 (y-1)}{q-p} }
    = g_{p,q,0}^{-1}(y)
    \leq \Phi_{p,q}^{-1}(y).
\]
Moreover, if $y \in [1,\Phi_{p,q}(1+\frac{\delta}{\max\{ |p|, |q| \} + 3/2})]$, then
\[
    \Phi_{p,q}^{-1}(y)
    \leq g_{p,q,\delta}^{-1}(y)
    = 1 + \sqrt{ \frac{8 (y-1) }{(1-\delta) (q-p)} }.
\]
This yields the required bounds to $\Phi_{p,q}^{-1}(1+\varepsilon)$ for (i).

For (ii) and (iii), we note that it suffices to consider the case for $p=-\infty$
since $\Phi_{-q,-p} = \Phi_{p,q}$ by Remark~\ref{rem:phi_pq}.
Now, the inverse function of $\Phi_{-\infty,q}$ restricted to $[1,\infty)$ can be given explicitly as
\[
    \Phi_{-\infty,q}^{-1}(y)
    = \begin{cases}
        (2 y^q - 1)^{\frac{1}{q}}, & \text{if } q \notin \{0,\infty\},
        \\
        y^2, & \text{if } q = 0,
        \\
        y, & \text{if } q = \infty.
    \end{cases}
\]
This immediately verifies (iii).
For (ii), we note that the mapping $\Phi_{-\infty,q}^{-1}$ can be extended to $y$ slightly smaller than $1$ to make it differentiable at any $y \in \Phi_{-\infty,q}([1,\infty))$.
We can thus use Taylor approximation to obtain
\[
    \Phi_{-\infty,q}^{-1}(y)
    = \Phi_{-\infty,q}^{-1}(1) + (\Phi_{-\infty,q}^{-1})'(1) \cdot (y-1) + R_{-\infty,q}(y)
    = 1 + 2 (y-1) + R_{-\infty,q}(y),
\]
where $R_{-\infty,q}(y) \coloneqq \Phi_{-\infty,q}^{-1}(y) - (1 + 2(y-1)) = \CO((y-1)^2)$.
This completes the proof.
\end{proof}

\section{Banach--Mazur Distance}
\label{sec:distance}

We close this paper by drawing connections between our bounds on covering radii between $p$-means and the Banach--Mazur distance.
Our focus hereby lies on identifying certain structures within the Banach--Mazur compactum that are provided by $p$-means.

Theorem~\ref{thm:upper_bounds}~(i) lets us estimate the Banach--Mazur distance between $K$ (or $L$)
and the $p$-means of $K$ and $L$.
The special case for $p$-symmetrizations, i.e., when $L = -K$, is of particular interest.
It extends the rather well-known result that for any $K \in \CK^n_0$,
$\frac{K-K}{2}$ minimizes the Banach--Mazur distance to $K$ among all symmetric convex bodies.
A generalization of this property to the four standard symmetrizations already appears in \cite[Remark~$4.6$]{BrDiMe}.

\begin{corollary}
\label{cor:dBM}
Let $p \in [-\infty,\infty]$ and $K,L,C \in \CK^n_0$ such that $\lowmean{p}{K}{L} \subset C \subset \upmean{p}{K}{L}$.
Then
\[
	d_{BM}(K,C) \leq R_0^{\max}(K,L).
\]
Moreover, if $K$ is Minkowski centered, $L = -K$, and $C$ symmetric, then
\[
	d_{BM}(K,C) = s(K).
\]
\end{corollary} \begin{proof}
\eqref{eq:reverse_chain} implies
\begin{align*}
	K
		\subset \conv(K \cup L)
		& \subset m_{-p}(R_0^{\max}(K,L),1) \lowmean{p}{K}{L}
		\subset m_{-p}(R_0^{\max}(K,L),1) C \\
		& \subset m_{-p}(R_0^{\max}(K,L),1) \upmean{p}{K}{L}
		\subset R_0^{\max}(K,L) \cdot (K \cap L)
		\subset R_0^{\max}(K,L) K
\end{align*}
and therefore immediately the upper bound.

If $K$ is Minkowski centered and $L = -K$, then $R_0^{\max}(K,L) = s(K)$.
Moreover, it is well-known that $d_{BM}(K,C) \geq s(K)$ if $C$ is symmetric (see, e.g.,~\cite{Grue}).
Together with the upper bound, the corollary is proven.
\end{proof}

\begin{remark}
With Remark~\ref{rem:R0_boundary}, we see that Corollary~\ref{cor:dBM} remains valid
if we replace for fulldimensional bodies $K$ and $L$ the assumption $0 \in \inte(K \cap L)$ by $0 \in K \cap L$ and $R_0^{\max}(K,L) \in (0,\infty)$.

Now, let $K'$ and $L'$ be fulldimensional affine transformations of $K$ and $L$ that satisfy
\[
	0 \in K' \cap L'
		\quad \text{and} \quad
	K' \subset \sqrt{d_{BM}(K,L)} L' \subset d_{BM}(K,L) K'.
\]
Such transformations always exist by Remark~\ref{rem:R0_basics}.
In this case, $R_0^{\max}(K',L') = \sqrt{d_{BM}(K,L)}$.
Thus, any convex body $C$ with $\lowmean{p}{K'}{L'} \subset C \subset \upmean{p}{K'}{L'}$ for some $p \in [-\infty,\infty]$ satisfies
\[
	d_{BM}(K,C) = d_{BM}(L,C) = \sqrt{d_{BM}(K,L)},
\]
where we used the extended Corollary~\ref{cor:dBM} and the triangle inequality.
Evidently, all $p$-means allow to construct bodies that are midpoints (in the multiplicative sense)
between $K$ and $L$ w.r.t.~the Banach--Mazur distance.
It is rather easy to see that this process can also be used to obtain geodesics between $K$ and $L$.

While the existence of geodesics between two bodies w.r.t.~the Banach--Mazur distance is known (cf.~\textup{\cite{ArKov}}),
it might be helpful in understanding the general structure of the Banach--Mazur compactum that they can be constructed using any $p$-mean.
For example, estimates on the distance between $p$-means could in turn yield bounds on the distances between the original bodies.
\end{remark}

If $K \in \CK^n_0$ is not 0-symmetric, then $\lowmean{p}{K}{-K} \neq \upmean{p}{K}{-K}$ for
$n \geq 2$ and $p \in \R$ by Theorem~\ref{thm:same_p}.
Thus, the relaxation to a set $C$ between these two symmetrizations in Corollary~\ref{cor:dBM}
generally applies to more sets than just the two means.

We also see that any non-symmetric (Minkowski centered) convex body is equidistant
w.r.t.~the Banach--Mazur distance from an infinitely large class of affinely non-equivalent symmetric bodies.
This seems to be of particular interest in the light of \cite{KobVa},
where it is shown that for $n = 2$ and $s \in [\frac{7}{4},2]$,
there exists a Minkowski centered body $K \in \CK^2_0$ with $s(K) = s$ that is at distance $s(K)$ from all symmetric bodies w.r.t.~the Banach--Mazur distance.

\begin{remark}
For a Minkowski centered body $K \in \CK^n_0$,
one can by Corollary~\ref{cor:dBM} continuously and increasingly deform $K \cap (-K)$ into $\conv(K \cup (-K))$
such that all intermediate bodies are symmetric and have Banach--Mazur distance $s(K)$ to $K$.
Since all intermediate bodies only need to be contained between an upper and a lower $p$-mean for some $p \in [-\infty,\infty]$ to apply the corollary,
there is even some degree of freedom in the construction of such a process.

This naturally raises the question whether any symmetric body $C$ with an affine transformation $C'$
satisfying $K \cap (-K) \subset C' \subset \conv(K \cup (-K))$ fulfills $d_{BM}(K,C) = s(K)$.
If this implication had been true, it might have helped with extending the result in \textup{\cite{KobVa}}
concerning bodies equidistant from all symmetric bodies to higher dimensions or smaller Minkowski asymmetry values.
One might also ask if the reverse implication is true,
as this could have had consequences on how small the asymmetry of such bodies can be.

Unfortunately, both implications turn out to be incorrect in general:
Consider a Minkowski centered regular triangle $T \in \CK^2_0$.
Let $s \in [\sqrt{4/3},2)$ and define $K = T \cap (-s T)$.
Since
\begin{equation}
\label{eq:dBM_example}
	\frac{s}{2} T
		\subset K
		\subset (s^2 T) \cap (-s T)
		= -s K
		\subset -s T
\end{equation}
and $s(T) = 2$, $K$ is Minkowski centered with $s(K) = s$.
It can be verified that the Euclidean circumcircle of $K \cap (-K) = T \cap (-T)$ is a subset of $\conv(K \cup (-K))$.
Now, $d_{BM}(K,\B_2) = s(K)$ would imply that there exist an ellipse $E$ and some $t \in \R^n$ with $K \subset E \subset -s K + t$.
From replacing $-s K$ with $-s K + t$ in \eqref{eq:dBM_example}
and the uniqueness of the Minkowski center of $T$,
we obtain $t = 0$.
However, $\bd(K) \cap \bd(-s K)$ contains a segment, which would thus also have to belong to $\bd(E)$, a contradiction.
Therefore, we have found an example where $K \cap (-K) \subset C \subset \conv( K \cup (-K) )$ does not imply $d_{BM}(K,C) = s(K)$.

For a counterexample to the reverse implication, we choose $C = \B_2$ and
\[
	K = \conv \left( \left\{
				\begin{pmatrix} 0\\ 6 \end{pmatrix},
				\begin{pmatrix} \pm 2 \sqrt{5} \\ 1\end{pmatrix},
				\begin{pmatrix} \pm 2 \sqrt{5} \\ -4 \end{pmatrix}
			\right\} \right).
\]
It is straightforward to verify $-K \subset 6 C \subset \frac{3}{2} K$,
where Proposition~\ref{prop:opt_cont} shows that $-K \subset \frac{3}{2} K$ is optimal.
It follows $d_{BM}(K,C) = s(K) = \frac{3}{2}$.
However, the boundaries of $K \cap (-K)$ and $\conv(K \cup (-K))$ contain a common segment,
so no ellipse fits between these two bodies.

Finally, let us briefly mention that the above questions can also be considered
for the inclusion between means in reverse order, i.e., between $\conv(K \cup (-K))$ and $s(K) (K \cap (-K))$.
While any symmetric body $C$ with an affine transformation $C'$ satisfying
\[
	-K
	\subset \conv(K \cup (-K))
	\subset C'
	\subset s(K) (K \cap (-K))
	\subset s(K) K
\]
clearly fulfills $d_{BM}(K,C) = s(K)$,
the reverse implication again fails in general as easily seen for $K$ a triangle and $C$ a parallelogram.
\end{remark}

\bigskip

René Brandenberg --
Technical University of Munich, Department of Mathematics, Germany. \\
\textbf{rene.brandenberg@tum.de}

Florian Grundbacher -- 
Technical University of Munich, Department of Mathematics, Germany. \\
\textbf{florian.grundbacher@tum.de}

\vfill\eject

\end{document}